\documentclass{article}
\usepackage{graphicx}
\usepackage[T1]{fontenc}
\usepackage{adjustbox}
\usepackage{comment}
\usepackage{amsmath}
\usepackage{amsthm}
\usepackage{amssymb}
\usepackage{adjustbox}
\usepackage{diagbox}
\usepackage{mathtools}
\usepackage{tikz, pgfplots}
\usetikzlibrary{positioning}
\usepackage{rotating}
\usepackage{lscape}
\usepackage{multicol}
\usepackage{multirow}
\usepackage[thinlines]{easytable}
\usepackage{cite}
\usepackage{multicol}

\newtheorem{theorem}{Theorem}

\newtheorem{lemma}[theorem]{Lemma}

\theoremstyle{definition}
\newtheorem{definition}{Definition}

\title{Computing the Haar state of $\mathcal{O}(SL_q(3))$ using value preserving (anti)homomorphisms}

\author{Ting Lu}

\date{}

\begin{document}

\maketitle
\begin{abstract}
In this paper, we introduce two (anti)homomorphisms that preserve the Haar state values of monomials. Together with the modular automorphism, the three (anti)homomorphisms are used in our new algorithm to compute the Haar states of monomials on $\mathcal{O}(SL_q(3))$. Comparing with the algorithm proposed in the author's previous work~\cite{lu2023}, the new algorithm reduces the linear relations used in the computation to a half.
\end{abstract}
\section{Introduction}
The Drinfeld-Jimbo~\cite{drinfeld1986quantum}~\cite{jimbo1985aq} quantum group $\mathcal{O}(SL_q(n))$ is a Hopf algebra introduced as a deformation of the algebra of coordinate functions on the Lie group $SL(n)$. The deformation relies on a parameter denoted as $q\in \mathbb{C}, 0<|q|<1$, in this paper and when $q\rightarrow 1$, the limiting Hopf algebra $\mathcal{O}(SL_1(n))$ is the algebra of coordinate functions on the Lie group $SL(n)$. It is possible to define a $^*$-Hopf algebra structure on $\mathcal{O}(SL_q(n))$ and the $^*$-Hopf algebra is denoted as $\mathcal{O}(SU_q(n))$~\cite{klimyk2012quantum}, a $q$-deformation of the algebra of coordinate functions on the Lie group $SU(n)$. Similar to the case of $\mathcal{O}(SL_1(n))$, $\mathcal{O}(SU_1(n))$ is the algebra of coordinate functions on the Lie group $SU(n)$. Moreover, the Haar state on $\mathcal{O}(SU_1(n))$ turns out to be the normalized integral with respect to the Haar measure on $SU(n)$. In this sense, the Haar state on $\mathcal{O}(SU_q(n))$ can be considered as the deformed Haar measure on $SU(n)$ and it is natural to seek for deformed versions of classical numerical results regarding the Haar measure on $SU(n)$.  Since the the Haar state of $\mathcal{O}(SU_q(n))$ can be considered as the Haar state on $\mathcal{O}(SL_q(n))$ extended to $\mathcal{O}(SU_q(n))$ by the $^*$ operation, this paper will study the computation of the Haar state on $\mathcal{O}(SL_q(n))$. As most of the interesting numerical results on $SU(n)$ are evaluating integrals of certain monomials (see Collins~\cite{collins2003moments} for example), we will compute the Haar state of $\mathcal{O}(SL_q(n))$ on a monomial basis.\\
\\
The existence and uniqueness of the Haar state on $\mathcal{O}(SL_q(n))$ relies on the fact that the quantum group is a co-semisimple Hopf algebra. The study of Haar state on co-semisimple Hopf algebra dates back to 1969 when Sweedler~\cite{sweedler1969hopf} showed that there is a unique "Haar state" up to normalization on any co-semisimple Hopf algebra. More precisely, Sweedler gives the Peter-Weyl decomposition of  a co-semisimple Hopf algebra and set the Haar state of the vector basis of the irreducible co-module, matrix coefficients, as zero for all matrix coefficients except the multiplicative unit $1$ which equals to $1$. However, the explicit expressions of matrix coefficients (in term of generators) are unknown except for the case $n=2$~\cite{klimyk2012quantum}. As a consequence, computing the Haar states of monomials using matrix coefficients is not an efficient method.\\
\\
In this paper, we develop an algorithm to compute the Haar state of these monomials using the following three (anti)homomorphisms:
\begin{definition}
    On $\mathcal{O}(SL_q(n))$, the\textbf{ Modular Automorphism}~\cite{noumi1993finite} $\eta$ is defined as:
    $$\eta(x_{i,j})=q^{2i+2j-2n-2}\cdot x_{i,j}$$
    where $x_{i,j}$'s are the generators of $\mathcal{O}(SL_q(n))$.
\end{definition}
\begin{definition}
    On $\mathcal{O}(SL_q(n))$, the \textbf{Diagonal Flip Homomorphism} $\gamma$ is defined as:
    $$\gamma(x_{i,j})=x_{j,i}.$$
\end{definition}
\begin{definition}
    On $\mathcal{O}(SL_q(n))$, the \textbf{Double Flip Anti-homomorphism} $\omega$ is defined as:
    $$\omega(x_{i,j})=x_{n+1-i,n+1-j}.$$
\end{definition}
\noindent We will prove the following theorems which are the key to reduce the computational load:
\begin{theorem}
    Define $\hat{h}=h\circ \gamma$. Then, $\hat{h}=h$ on $\mathcal{O}(SL_q(n))$.
\end{theorem}
\begin{theorem}
    Define $\Tilde{h}=h\circ \omega$. Then, $\Tilde{h}=h$ on $\mathcal{O}(SL_q(n))$.
\end{theorem}

\noindent To keep consistency with the author's previous work, we denote the generators of $\mathcal{O}(SL_q(3))$ as:
\begin{equation*}
\begin{matrix}
a&b&c\\
d&e&f\\
g&h&k.
\end{matrix}
\end{equation*}
Recall that standard monomials of order $m=c_1+c_2+c_3+c_4+c_5+c_6$ are in the form:
$$(aek)^{c_1}(afh)^{c_2}(bdk)^{c_3}(bfg)^{c_4}(cdh)^{c_5}(ceg)^{c_6}.$$
and segments $aek$, $afh$, and $bdk$ are defined as {\bf high-complexity segments}; segments $bfg$, $cdh$, and $ceg$ are defined as {\bf low-complexity segments}. The monomial basis we use in this paper is the same as Lu~\cite{lu2023}:
$$\{(aek)^{m_1}(afh)^{m_2}(bdk)^{m_3}(bfg)^{m_4}(cdh)^{m_5}(ceg)^{m_6}, m_i\in\mathbf{N}_0 \text{ and } m_2\cdot m_3\cdot m_6=0\}.$$ By applying the three (anti)homomorphisms, we will show that ($0\le r,s,i,j$):
\newpage
\begin{theorem}[Main Theorem] The following equation hold on $\mathcal{O}(SL_q(3))$:
\begin{enumerate}
    \item[1)] \begin{equation}\label{symmetric_re1}
    \begin{split}
        &h\left((afh)^{r}(bdk)^{s}(bfg)^{i}(cdh)^{j}(ceg)^{m-r-s-i-j}\right)\\
        =&h\left((afh)^{s}(bdk)^{r}(bfg)^{j}(cdh)^{i}(ceg)^{m-r-s-i-j}\right)\\
        =&h\left((afh)^{s}(bdk)^{r}(bfg)^{i}(cdh)^{j}(ceg)^{m-r-s-i-j}\right)\\
        =&h\left((afh)^{r}(bdk)^{s}(bfg)^{j}(cdh)^{i}(ceg)^{m-r-s-i-j}\right)\\
        \end{split}
    \end{equation}
    \item[2)] \begin{equation}\label{symmetric_re2}
        \begin{split}
            &h\left((aek)^u(afh)^v(bdk)^w(bfg)^s(cdh)^r(ceg)^{m-u-v-w-r-s}\right)\\
            =&h\left((aek)^u(afh)^v(bdk)^w(bfg)^r(cdh)^s(ceg)^{m-u-v-w-r-s}\right)
        \end{split}
    \end{equation}
    \item[3)] \begin{equation}\label{symmetric_re3}
        \begin{split}
            &h\left((aek)^u(afh)^v(bdk)^w\right)\\
            =&h\left((aek)^u(afh)^w(bdk)^v\right)
        \end{split}
    \end{equation}
\end{enumerate}
\end{theorem}
\noindent Equation (\ref{symmetric_re1}), (\ref{symmetric_re2}), and (\ref{symmetric_re3}) imply that we only need to compute the Haar state of one of the monomials in these symmetric relations. Thus, the computational load of the algorithm proposed in this paper is largely reduced comparing to that of the algorithm in Lu~\cite{lu2023}. We also derive the explicit recursive relation between standard monomials in the form of $(cdh)^r(bfg)^t(ceg)^{m-r-t}$ using the modular automorphism.\\
\\
In this paper, we will assume that the Haar states of monomials appearing in the Source Matrix and monomials in the form of $(cdh)^i(ceg)^{m-i}$ with $1\le i\le m$ are known~\cite{lu2023}. We will use the method proposed in Lu~\cite{lu2023} Section 2.3 to construct linear relations between the Haar states of our monomial basis. The algorithm we proposed in this paper is as follow:
\begin{enumerate}
    \item[Step 1)]Compute the Haar state of monomials in the form of $(cdh)^r(bfg)^s(ceg)^{m-r-t}$ by the following recursive relation:
    \begin{equation*}
    \begin{split}
        &\frac{q^2(q^{m-s}-q^{s-m})^2}{(1-q^2)^2}\cdot h\left((cdh)^r(bfg)^{s+1}(ceg)^{m-r-s-1}\right)\\
        =&-\frac{q}{q^2-1}\cdot h\left((cdh)^r(bfg)^{s}(ceg)^{m-r-s}\right)\\
        &-\sum_{i=0}^{s-1}a_i\cdot h\left((cdh)^r(bfg)^{s-i}(ceg)^{m-r-s+i}\right)\\
        &-(q^{-1}-q)^{s-1}q^{-2m+2s}\cdot h\left((cdh)^r(ceg)^{m-r}\right)
    \end{split}
\end{equation*}
where
\begin{equation*}
    a_i=(q^{-1}-q)^{i-1}{s+1\choose i+1}q^{-2m+2s}+(q-q^{-1})^{i-1}q^{2i-2}{s\choose i+1}q^{2m-2s+4}
\end{equation*}
    \item[Step 2)] Compute the Haar state of monomials in the form of\\ $(afh)^w(bfg)^s(cdh)^r(ceg)^{m-w-r-s}$ by a recursive relation on the number of generator $a$:
    \begin{enumerate}
        \item[i)] if $m-w-r-s\ge 1$, then use the equation:
        \begin{equation*}
            \begin{split}
                & h\left((afh)^w(bfg)^s(cdh)^r(ceg)^{m-w-r-s}\right)\\
                =&\sum_{i=0}^{w-1}\frac{q^{2i}(1-q^2)}{1-q^{2w}}\cdot h\left((afh)^ibfgcdh(afh)^{w-1-i}(cdh)^r(ceg)^{m-w-r-s-1}\right)
            \end{split}
        \end{equation*}
        \item[ii)] if $m-w-r-s=0$, then compute the Haar state of  $(afh)^w(bfg)^s(cdh)^r$ by the linear relation derived from equation basis\\ $(afh)^{w-1}(cdh)^{r+1}(bfg)^{m-r-w}$ with comparing basis $(aek)^{m-1}bdk$.
    \end{enumerate}
    \item[Step 3)] Find the Haar state of $(bdk)^w(bfg)^s(cdh)^r(ceg)^{m-w-r-s}$ using Equation (\ref{symmetric_re1}).
    \item[Step 4)] Compute the Haar state of monomials in the form of\\ $(aek)^u(afh)^v(bdk)^w(bfg)^s(cdh)^r(ceg)^{m-u-v-w-r-s}$ by a recursive relation on the total number of high complexity segments $u+v+w$:
    \begin{enumerate}
        \item[i)] if $u=0$, $m-v-w-r-s\ge 1$, use the relation:
        \begin{equation*}
        \begin{split}
            afhbdkceg =& q*aekbfgcdh+(1 - q^2)*aekbfgceg\\
        &+(1 - q^2)*aekcdhceg+(q^2 - 1)^2/q*aek(ceg)^2\\
        &+(1 - q^2)*afhbfgcdh+(q^3 - q)*afhbfgceg\\
        &+(q^3 - q)*afhcdhceg-(q^2 - 1)^2*afh(ceg)^2
        \end{split}
    \end{equation*}
    to rewrite $(afh)^v(bdk)^w(bfg)^s(cdh)^r(ceg)^{m-v-w-r-s}$ as a linear combination of standard monomials with at most $v+w-1$ high complexity segments.
    \item[ii)] if $u=0$, $m-v-w-r-s=0$, then use the linear relation derived from equation basis $(afh)^v(bdk)^{w-1}(bfg)^{s+1}(cdh)^{m-v-w-s}$ and comparing basis $(aek)^{m-1}afh$ to compute the Haar state of $(afh)^v(bdk)^{w}(bfg)^{s}(cdh)^{m-v-w-s}$.
    \item[iii)] if $u\ge 1$, apply the inductive algorithm in Section 6.2.3.
    \end{enumerate}
\end{enumerate}
In the above algorithm, Equation (\ref{symmetric_re1}) takes effect in Step 1), Step 2), Step 3), Step 4) i) and Step 4) ii) and Equation (\ref{symmetric_re2}) and (\ref{symmetric_re3}) take effect in Step 4 iii). In the algorithm in Lu~\cite{lu2023}, then number of linear equations in the algorithm is the same as the number of standard monomials of order $m$. In this algorithm, the number of linear equations is only a half of the number of standard monomials of order $m$. Hence, the number of linear equations used in this algorithm reduces to a half of that used in the algorithm in Lu~\cite{lu2023}.

\section{Three (anti-)homomorphisms that preserve the Haar state and the proof of the main theorem}
\subsection{The Modular Automorphism on $\mathcal{O}(SL_q(n))$}
In this subsection, we follow NYM's paper~\cite{noumi1993finite} to introduce the \textit{\textbf{modular automorphism}}. Although NYM introduced the modular automorphism on $\mathcal{O}(SU_q(n))$, but their results are directly applicable on $\mathcal{O}(SL_q(n))$.

\hfill

\noindent The modular automorphism on $\mathcal{O}(SL_q(n))$, $\eta$, is defined as:
$$\eta(x_{i,j})=q^{2i+2j-2n-2}x_{i,j}.$$
The modular automorphism satisfies:
$$h(x\cdot y)=h(y\cdot \eta(x))$$
for all $x,y\in \mathcal{O}(SL_q(n))$. By direct computation, we can show that for every standard monomial $x_{\sigma}=\prod_{i=1}^nx_{i,\sigma(i)}$ of order $1$:
$$\eta(x_\sigma)=x_\sigma.$$

\subsection{The diagonal flip homomorphism on $\mathcal{O}(SL_q(n))$}
In this section, we prove Theorem 1. We start with the following lemma:
\begin{lemma}
    Let $\tau$ be the homomorphism that flip the left and right component of a tensor product. The followings are true:
    \begin{enumerate}
        \item[1.] $\gamma(D_q)=D_q$, or in equivalence, $\gamma(1)=1$
        \item[2.] $\Delta\circ\gamma=\tau\circ(\gamma\otimes\gamma)\circ\Delta$
    \end{enumerate}
\end{lemma}
\begin{proof}

\hfill

    \begin{enumerate}
        \item[1.] First, we recall a simple fact from NYM~\cite{noumi1993finite}: if $\tau,\sigma$ are two permutations on $n$ letters and $\tau=\sigma^{-1}$, then 
        $$x_{\sigma(1),1}\cdots x_{\sigma(n),n}=x_{1,\tau(1)}\cdots x_{n,\tau(n)}$$
        and $l(\sigma)=l(\tau)$, where $l(\sigma)$ is the inversion number of permutation $\sigma$. Hence, we get:
        \begin{equation*}
            \begin{split}
                &\gamma\left(x_{1,\tau(1)}\cdots x_{n,\tau(n)}\right)=x_{\tau(1),1}\cdots x_{\tau(n),n}=x_{1,\sigma(1)}\cdots x_{n,\sigma(n)}.
            \end{split}
        \end{equation*}
        Applying the relation to $\gamma(D_q)$, we find that:
        \begin{equation*}
            \begin{split}
                \gamma(D_q)=&\sum_{\tau\in S_n}(-q)^{l(\tau)}\cdot \gamma\left(x_{1,\tau(1)}\cdots x_{n,\tau(n)}\right)\\
                =&\sum_{\tau\in S_n}(-q)^{l(\tau)}\cdot x_{1,\tau^{-1}(1)}\cdots x_{n,\tau^{-1}(n)}\\
                =&\sum_{\tau\in S_n}(-q)^{l(\tau^{-1})}\cdot x_{1,\tau^{-1}(1)}\cdots x_{n,\tau^{-1}(n)}\\
                =&D_q
            \end{split}
        \end{equation*}
        \item[2.] \begin{equation*}
    \begin{split}
        \Delta\circ\gamma(x_{i,j})&=\Delta(x_{j,i})=\sum_{k=1}^nx_{j,k}\otimes x_{k,i}\\
        =&\tau\left(\sum_{k=1}^nx_{k,i}\otimes x_{j,k}\right)\\
        =&\tau\circ(\gamma\otimes\gamma)\left(\sum_{k=1}^nx_{i,k}\otimes x_{k,j}\right)\\
        =&\tau\circ(\gamma\otimes\gamma)\circ\Delta(x_{i,j})
    \end{split}
\end{equation*}
    \end{enumerate}
\end{proof}
\noindent\textbf{Theorem 2}: \textit{Define }$\hat{h}=h\circ\gamma$. \textit{Then,} $\hat{h}=h$ \textit{on} $\mathcal{O}(SL_q(n))$.
\begin{proof}
    Let $x$ be a monomial in $\mathcal{O}(SL_q(n))$. We may write $\Delta(x)=\sum x_{(1)}\otimes x_{(2)}$ By the left translation invariant property of $h$, we get:
    \begin{equation*}
        \begin{split}
            \hat{h}(x)\cdot 1=&h\circ\gamma(x)\cdot 1=(id\otimes h)\circ\Delta\circ\gamma(x)\\
            =&(id\otimes h)\circ\tau\circ(\gamma\otimes\gamma)\circ\Delta(x)\\
            =&(h\otimes id)\circ(\gamma\otimes\gamma)\circ\Delta(x)\\
            =&\sum\hat{h}(x_{(1)})\cdot \gamma(x_{(2)})
        \end{split}
    \end{equation*}
    Applying $\gamma$ on both sides of the equation, we find that:
    $$\hat{h}(x)\cdot 1=\hat{h}(x)\cdot \gamma(1)=\sum\hat{h}(x_{(1)})\cdot x_{(2)}=(\hat{h}\otimes id)\circ\Delta(x).$$
    In other words, $\hat{h}$ has the right translation invariant property. Similarly, using the right translation invariant property of $h$, we can show the left translation invariant property of $\hat{h}$. Finally, we have $\hat{h}(1)=h(1)=1$. Thus, by the uniqueness of the Haar state on $\mathcal{O}(SL_q(n))$, we must have $\hat{h}=h$.
\end{proof}

\subsection{The double flip anti-homomorphism on $\mathcal{O}(SL_q(n))$}
In this section, we prove Theorem 2. We start with the following lemma:
\begin{lemma} \label{Lemma 5}
The following equations hold:
\begin{enumerate}
    \item[1)] $\omega(D_q)=D_q$ or in equivalence,  $\omega(1)=1$
    \item[2)] $\Delta\circ\omega=(\omega\otimes\omega)\circ\Delta$
\end{enumerate}
\end{lemma}
\begin{proof}
\begin{enumerate}
    \item[1)]  $\omega$ maps a standard monomial $\prod_{i=1}^n x_{i,\sigma(i)}$ to:
$$\omega \left(\prod_{i=1}^n x_{i,\sigma(i)}\right)=\prod_{n+1-i=1}^nx_{n+1-i,n+1-\sigma(i)}$$
which is another standard monomial corresponding to the permutation that send $n+1-i$ to $n+1-\sigma(i)$. Denote this permutation as $\omega\cdot\sigma$. It is easy to check that $\omega\cdot\omega\cdot\sigma=\sigma$. Thus, $\sigma$ and $\omega\cdot\sigma$ is one to one. Notice that if $i<j$ and $\sigma(i)>\sigma(j)$, then $n+1-i>n+1-j$ and $n+1-\sigma(i)<n+1-\sigma(j)$ and vice versa. In other word, there is a one to one correspondence between the inversions of $\sigma$ and the inversions of $\omega\cdot\sigma$. Therefore, we have:
\begin{equation*}
    \begin{split}
        \omega(1)=\omega(D_q)=&\omega\left(\sum_{\sigma\in S_n}(-q)^{l(\sigma)}\prod_{i=1}^n x_{i,\sigma(i)}\right)\\
        =&\sum_{\sigma\in S_n}(-q)^{l(\omega\cdot\sigma)}\prod_{j=1}^n x_{j,\omega\cdot\sigma(j)}\\
        =&D_q=1
    \end{split}
\end{equation*}
    \item[2)] It suffice to verify the equation on the generators of $\mathcal{O}(SL_q(n))$: 
    \begin{equation*}
        \begin{split}
            \Delta\circ\omega(x_{i,j})&=\Delta(x_{n+1-i,n+1-j})=\sum_{k=1}^n x_{n+1-i,k}\otimes x_{k,n+1-j}\\
            &=\sum_{k=1}^n x_{n+1-i,n+1-k}\otimes x_{n+1-k,n+1-j}\\
            &=\sum_{k=1}^n \omega(x_{i,k})\otimes \omega(x_{k,j})=(\omega\otimes\omega)\circ\Delta(x_{i,j})
        \end{split}
    \end{equation*}
\end{enumerate}
\end{proof}

\hfill

\noindent\textbf{Theorem 2}: \textit{Define} $\Tilde{h}=h\circ \omega$. \textit{Then, }$\Tilde{h}=h$ \textit{on }$\mathcal{O}(SL_q(n))$.
\begin{proof}
    By the left translation invariant property, we get:
\begin{equation*}
    (id\otimes h)\Delta(\omega(a))=h(\omega(a))\cdot 1
\end{equation*}
If we write $\Delta(a)=\sum a_{(1)}\otimes a_{(2)}$, then the left translation invariant property implies that:
\begin{equation*}
\begin{split}
    h(\omega(a))\cdot 1&=(id\otimes h)\Delta(\omega(a))\\
    &=(id\otimes h)\circ(\omega\otimes\omega)(\Delta(a))\\
    &=\sum h(\omega(a_{(2)}))\cdot\omega(a_{(1)})\\
    &=\sum \Tilde{h}(a_{(2)})\cdot\omega(a_{(1)})\\
\end{split}
\end{equation*}
If we apply $\omega$ on both sides of $\Tilde{h}(a)\cdot 1=\sum \Tilde{h}(a_{(2)})\cdot\omega(a_{(1)})$, we get:
$$\sum \Tilde{h}(a_{(2)})\cdot a_{(1)}=\Tilde{h}(a)\cdot 1.$$
This is equivalent to say that:
$$(id\otimes \Tilde{h})\Delta(a)=\Tilde{h}(a)\cdot 1$$
Similarly, we can show that:
$$(\Tilde{h}\otimes id)\Delta(a)=\Tilde{h}(a)\cdot 1$$
and $\Tilde{h}(1)=1$. Then, by the uniqueness of the Haar state, we know that $\Tilde{h}$ is identical to $h$ on $\mathcal{O}(SL_q(n))$.
\end{proof}

\subsection{Proof of the main theorem}
We split the main theorem into 3 parts.\\
\noindent\textbf{Theorem 3 part 1)}
    \begin{equation*}
    \begin{split}
        &h\left((afh)^{r}(bdk)^{s}(bfg)^{i}(cdh)^{j}(ceg)^{m-r-s-i-j}\right)\\
        =&h\left((afh)^{s}(bdk)^{r}(bfg)^{j}(cdh)^{i}(ceg)^{m-r-s-i-j}\right)\\
        =&h\left((afh)^{s}(bdk)^{r}(bfg)^{i}(cdh)^{j}(ceg)^{m-r-s-i-j}\right)\\
        =&h\left((afh)^{r}(bdk)^{s}(bfg)^{j}(cdh)^{i}(ceg)^{m-r-s-i-j}\right)\\
    \end{split}
\end{equation*}
\begin{proof}

\hfill

    \begin{enumerate}
        \item[1)] First, we show that:
    \begin{equation*}
    \begin{split}
        &h\left((afh)^{r}(bdk)^{s}(bfg)^{i}(cdh)^{j}(ceg)^{m-r-s-i-j}\right)\\
        =&h\left((afh)^{s}(bdk)^{r}(bfg)^{j}(cdh)^{i}(ceg)^{m-r-s-i-j}\right).
    \end{split}
    \end{equation*}
    We have:
    \begin{equation*}
        \begin{split}
            &h\left((afh)^{r}(bdk)^{s}(bfg)^{i}(cdh)^{j}(ceg)^{m-r-s-i-j}\right)\\
            =&\Tilde{h}\left((afh)^{r}(bdk)^{s}(bfg)^{i}(cdh)^{j}(ceg)^{m-r-s-i-j}\right)\\
            =&h\left((ceg)^{m-r-s-i-j}(bfg)^{j}(cdh)^{i}(afh)^{s}(bdk)^{r}\right)\\
            =&h\left((afh)^{s}(bdk)^{r}\eta\left((ceg)^{m-r-s-i-j}(bfg)^{j}(cdh)^{i}\right)\right)\\
             =&h\left((afh)^{s}(bdk)^{r}(bfg)^{j}(cdh)^{i}(ceg)^{m-r-s-i-j}\right)\\
        \end{split}
    \end{equation*}
    \end{enumerate}
    \item[2)] Then, we show that:
    \begin{equation*}
    \begin{split}
        &h\left((afh)^{s}(bdk)^{r}(bfg)^{j}(cdh)^{i}(ceg)^{m-r-s-i-j}\right)\\
        =&h\left((afh)^{s}(bdk)^{r}(bfg)^{i}(cdh)^{j}(ceg)^{m-r-s-i-j}\right).
    \end{split}
\end{equation*}
    We have:
    \begin{equation*}
    \begin{split}
        &h\left((afh)^{s}(bdk)^{r}(bfg)^{j}(cdh)^{i}(ceg)^{m-r-s-i-j}\right)\\
        =&\hat{h}\left((afh)^{s}(bdk)^{r}(bfg)^{j}(cdh)^{i}(ceg)^{m-r-s-i-j}\right)\\
        =&h\left((ahf)^{s}(dbk)^{r}(dhc)^{j}(gbf)^{i}(ceg)^{m-r-s-i-j}\right)\\
        =&h\left((afh)^{s}(bdk)^{r}(bfg)^{i}(cdh)^{j}(ceg)^{m-r-s-i-j}\right).
    \end{split}
\end{equation*}
    \item[3)] Finally, we show that:
    \begin{equation*}
    \begin{split}
        &h\left((afh)^{s}(bdk)^{r}(bfg)^{i}(cdh)^{j}(ceg)^{m-r-s-i-j}\right)\\
        =&h\left((afh)^{r}(bdk)^{s}(bfg)^{j}(cdh)^{i}(ceg)^{m-r-s-i-j}\right).
    \end{split}
\end{equation*}
We have:
\begin{equation*}
    \begin{split}
        &h\left((afh)^{s}(bdk)^{r}(bfg)^{i}(cdh)^{j}(ceg)^{m-r-s-i-j}\right)\\
        =&\Tilde{h}\left((afh)^{s}(bdk)^{r}(bfg)^{i}(cdh)^{j}(ceg)^{m-r-s-i-j}\right)\\
        =&h\left((ceg)^{m-r-s-i-j}(bfg)^{j}(cdh)^{i}(afh)^{r}(bdk)^{s}\right)\\
        =&h\left((afh)^{r}(bdk)^{s}\eta((ceg)^{m-r-s-i-j}(bfg)^{j}(cdh)^{i})\right)\\
        =&h\left((afh)^{r}(bdk)^{s}(bfg)^{j}(cdh)^{i}(ceg)^{m-r-s-i-j}\right)
    \end{split}
\end{equation*}
\end{proof}
\noindent\textbf{Theorem 3 part 2)}
\begin{equation*}
    \begin{split}
        &h\left((aek)^u(afh)^v(bdk)^w(bfg)^s(cdh)^r(ceg)^{m-u-v-w-r-s}\right)\\
        =&h\left((aek)^u(afh)^v(bdk)^w(bfg)^r(cdh)^s(ceg)^{m-u-v-w-r-s}\right)
    \end{split}
\end{equation*}
\begin{proof}
    \begin{equation*}
        \begin{split}
            &h\left((aek)^u(afh)^v(bdk)^w(bfg)^s(cdh)^r(ceg)^{m-u-v-w-r-s}\right)\\
            =&\hat{h}\left((aek)^u(afh)^v(bdk)^w(bfg)^s(cdh)^r(ceg)^{m-u-v-w-r-s}\right)\\
            =&h\left((aek)^u(ahf)^v(dbk)^w(dhc)^s(gbf)^r(ceg)^{m-u-v-w-r-s}\right)\\
            =&h\left((aek)^u(afh)^v(bdk)^w(bfg)^r(cdh)^s(ceg)^{m-u-v-w-r-s}\right)
        \end{split}
    \end{equation*}
\end{proof}
\noindent\textbf{Theorem 3 part 3)}
\begin{equation*}
    \begin{split}
        &h\left((aek)^u(afh)^v(bdk)^w\right)\\
        =&h\left((aek)^u(afh)^w(bdk)^v\right)
    \end{split}
\end{equation*}
\begin{proof}
    \begin{equation*}
        \begin{split}
            &h\left((aek)^u(afh)^v(bdk)^w\right)\\
            =&\Tilde{h}\left((aek)^u(afh)^v(bdk)^w\right)\\
            =&h\left((afh)^w(bdk)^v(aek)^u\right)\\
            =&h\left((aek)^u\eta\left((afh)^w(bdk)^v\right)\right)\\
            =&h\left((aek)^u(afh)^w(bdk)^v\right)
        \end{split}
    \end{equation*}
\end{proof}

\section{The recursive relation for the Haar state of standard monomials in the form of $(cdh)^r(bfg)^t(ceg)^{m-r-t}$}\label{recuresive_cdhbfg}
In this section we will start our new algorithm assuming that the solution to the Source Matrix is given and the recursive relation of standard monomials in the form of $(cdh)^r(ceg)^{m-r}$ is provided~\cite{lu2023}. Before we discuss the new recursive relation, we make the following observation. Notice that:
$$cegafh=q^2*afhceg+(1-q^2)*bfgcdh$$
If we evaluate the Haar state on both sides of the equation and recall that $h\left(cegafh\right)=h\left(afh\cdot\eta(ceg)\right)=h\left(afhceg\right)$, we get:
$$h\left(afhceg\right)=q^2\cdot h\left(afhceg\right)+(1-q^2)\cdot h\left(bfgcdh\right)$$
This implies that $h\left(afhceg\right)=h\left(bfgcdh\right)$. Using similar trick, we have \\$h\left(cegafh(bfg)^{m_1}(cdh)^{m_2}(ceg)^{m_3}\right)=h\left(afh(bfg)^{m_1}(cdh)^{m_2}(ceg)^{m_3+1}\right)$ and we deduce that
\begin{equation}
    h\left(afh(bfg)^{m_1}(cdh)^{m_2}(ceg)^{m_3+1}\right)=h\left((bfg)^{m_1+1}(cdh)^{m_2+1}(ceg)^{m_3}\right).
    \label{eq:re}
\end{equation}
Next, we start to compute the recursive relation. From the previous section, we have computed the Haar state of $(cdh)^r(ceg)^{m-r}$ and $(bfg)^s(ceg)^{m-r}$. Thus, we will fix the index $r$ and derive a recursive relation in the index $s$.

\subsection{Analysis}
To compute the Haar state of $(cdh)^r(bfg)^{s+1}(ceg)^{m-r-s-1}$ with the restriction $m-r-s\ge 1$ and $r\ge 1$, we consider the linear relation derived from equation basis $(cdh)^r(bfg)^s(ceg)^{m-r-s}$ with comparing basis $(aek)^{m-1}bdk$. For the detail of this construction, see Lu~\cite{lu2023} Section 2.3. In the comultiplication of $(cdh)^r(bfg)^s(ceg)^{m-r-s}$, left components containing $(aek)^{m-1}bdk$ are in the following form:
\begin{enumerate}
    \item[1)] $(aek)^lbdk(aek)^{m-1-l}$
    \item[2)] $(aek)^lbek(aek)^kadk(aek)^{m-2-l-k}$
    \item[3)] $(aek)^ladk(aek)^kbek(aek)^{m-2-l-k}$ 
\end{enumerate}
\noindent When the comparing component is in the form $(aek)^lbdk(aek)^{m-1-l}$, the coefficient of $(aek)^{m-1}bdk$ in the decomposition of $(aek)^lbdk(aek)^{m-1-l}$ is $1$ and the corresponding relation components are:
\begin{enumerate}
    \item[1)] $(cdh)^kfah(cdh)^{r-1-k}(bfg)^s(ceg)^{m-r-s}$ 
    \item[2)] $(cdh)^r(bfg)^kecg(bfg)^{s-k-1}(ceg)^{m-r-s}$
    \item[3)] $(cdh)^r(bfg)^s(ceg)^kfbg(ceg)^{m-r-s-1-k}$ 
\end{enumerate}

\hfill

\noindent When the comparing component is in the form $(aek)^lbek(aek)^kadk(aek)^{m-2-l-k}$, the coefficient of $(aek)^{m-1}bdk$ in the decomposition of $(aek)^lbek(aek)^kadk(aek)^{m-2-l-k}$ is $1$ and the corresponding relation components are:
\begin{enumerate}
    \item[4)] $(cdh)^kfdh(cdh)^lcah(cdh)^{r-k-l-2}(bfg)^s(ceg)^{m-r-s}$ 
    \item[5)] $(cdh)^kfdh(cdh)^{r-k-1}(bfg)^lbcg(bfg)^{s-l-1}(ceg)^{m-r-s}$ 
    \item[6)]  $(cdh)^kfdh(cdh)^{r-k-1}(bfg)^{s}(ceg)^lcbg(ceg)^{m-r-s-l-1}$
    \item[7)] $(cdh)^r(bfg)^kefg(bfg)^lbcg(bfg)^{s-k-l-2}(ceg)^{m-r-s}$ 
    \item[8)]  $(cdh)^r(bfg)^kefg(bfg)^{s-k-1}(ceg)^lcbg(ceg)^{m-r-s-l-1}$ 
    \item[9)] $(cdh)^r(bfg)^s(ceg)^kfeg(ceg)^lcbg(ceg)^{m-r-s-l-k-2}$  
\end{enumerate}

\hfill

\noindent When the comparing component is in the form $(aek)^ladk(aek)^kbek(aek)^{m-2-l-k}$, the coefficient of $(aek)^{m-1}bdk$ in the decomposition of $(aek)^ladk(aek)^kbek(aek)^{m-2-l-k}$ is $q^2$ and the corresponding relation components are:
\begin{enumerate}
    \item[10)] $(cdh)^kcah(cdh)^lfdh(cdh)^{r-k-l-2}(bfg)^s(ceg)^{m-r-s}$ 
    \item[11)] $(cdh)^kcah(cdh)^{r-k-1}(bfg)^lefg(bfg)^{s-l-1}(ceg)^{m-r-s}$ 
    \item[12)]  $(cdh)^kcah(cdh)^{r-k-1}(bfg)^{s}(ceg)^lfeg(ceg)^{m-r-s-l-1}$
    \item[13)] $(cdh)^r(bfg)^kbcg(bfg)^lefg(bfg)^{s-k-l-2}(ceg)^{m-r-s}$ 
    \item[14)]  $(cdh)^r(bfg)^kbcg(bfg)^{s-k-1}(ceg)^lfeg(ceg)^{m-r-s-l-1}$ 
    \item[15)] $(cdh)^r(bfg)^s(ceg)^kcbg(ceg)^lfeg(ceg)^{m-r-s-l-k-2}$  
\end{enumerate}
Case 1): Using the modular automorphism and Equation (\ref{eq:re}), the Haar state of the monomial in case 1) can be transformed into:
\begin{equation*}
    \begin{split}
        &h\left((cdh)^kfah(cdh)^{r-1-k}(bfg)^s(ceg)^{m-r-s}\right)\\
        =&h\left((cdh)^{r}(bfg)^{s+1}(ceg)^{m-r-s-1}\right)-(q-q^{-1})h\left((cdh)^{r}(bfg)^s(ceg)^{m-r-s}\right)
    \end{split}
\end{equation*}

\hfill

\noindent Case 2): This case can be transformed into:
$$(cdh)^r(bfg)^{s-1}(ceg)^{m-r-s+1}$$

\hfill

\noindent Case 3): This case can be transformed into:
$$(cdh)^r(bfg)^{s+1}(ceg)^{m-r-s-1}-(q-q^{-1})*(cdh)^r(bfg)^{s}(ceg)^{m-r-s}$$

\hfill

\noindent Case 4): We have:
\begin{equation*}
    \begin{split}
        &(cdh)^kfdh(cdh)^lcah(cdh)^{r-k-l-2}(bfg)^s(ceg)^{m-r-s}\\
        =&q^{-2(l+1)}*(cdh)^{k+l+1}fah(cdh)^{r-k-l-2}(bfg)^s(ceg)^{m-r-s}
    \end{split}
\end{equation*}
Then, apply the modular automorphism and Equation (\ref{eq:re}), we get:
\begin{equation*}
    \begin{split}
        &q^{-2(l+1)}*h\left((cdh)^{k+l+1}fah(cdh)^{r-k-l-2}(bfg)^s(ceg)^{m-r-s}\right)\\
        =&q^{-2(l+1)}*h\left((cdh)^{r}(bfg)^{s+1}(ceg)^{m-r-s-1}\right)\\
        &-q^{-2(l+1)}(q-q^{-1})*h\left((cdh)^{r}(bfg)^s(ceg)^{m-r-s}\right)
    \end{split}
\end{equation*}

\hfill

\noindent Case 5): We have:
\begin{equation*}
    \begin{split}
        &(cdh)^kfdh(cdh)^{r-k-1}(bfg)^lbcg(bfg)^{s-l-1}(ceg)^{m-r-s}\\
        =&q^{-2(r-k)+1}\sum_{i=0}^{l+1}(q^{-1}-q)^i{l+1 \choose i}*(cdh)^{r}(bfg)^{s-i}(ceg)^{m-r-s+i}
    \end{split}
\end{equation*}

Case 6): We have:
\begin{equation*}
    \begin{split}
        &(cdh)^kfdh(cdh)^{r-k-1}(bfg)^{s}(ceg)^lcbg(ceg)^{m-r-s-l-1}\\
        =&q^{-2l-2r+2k}\sum_{i=0}^{s+1}(q^{-1}-q)^i{s+1\choose i}*(cdh)^{r}(bfg)^{s+1-i}(ceg)^{m-r-s-1+i}
    \end{split}
\end{equation*}

\hfill

\noindent Case 7): 
\begin{equation*}
    \begin{split}
        &(cdh)^r(bfg)^kefg(bfg)^lbcg(bfg)^{s-k-l-2}(ceg)^{m-r-s}\\
        =&\sum_{i=0}^{l+1}(q^{-1}-q)^i{l+1\choose i}*(cdh)^r(bfg)^{s-1-i}(ceg)^{m-r-s+1+i}
    \end{split}
\end{equation*}

\hfill

\noindent Case 8):
\begin{equation*}
    \begin{split}
        &(cdh)^r(bfg)^kefg(bfg)^{s-k-1}(ceg)^lcbg(ceg)^{m-r-s-l-1}\\
        =&q^{-2l-1}\sum_{i=0}^{s-k}(q^{-1}-q)^i{s-k\choose i}*(cdh)^r(bfg)^{s-i}(ceg)^{m-r-s+i}
    \end{split}
\end{equation*}

\hfill

\noindent Case 9): We have:
\begin{equation*}
    \begin{split}
        &(cdh)^r(bfg)^s(ceg)^kfeg(ceg)^lcbg(ceg)^{m-r-s-l-k-2}\\
        =&q^{-2l-2}*(cdh)^r(bfg)^{s+1}(ceg)^{m-r-s-1}\\
        &-q^{-2l-2}(q-q^{-1})*(cdh)^r(bfg)^{s}(ceg)^{m-r-s}
    \end{split}
\end{equation*}

\hfill

\noindent Case 10): We have:
\begin{equation*}
    \begin{split}
        &(cdh)^kcah(cdh)^lfdh(cdh)^{r-k-l-2}(bfg)^s(ceg)^{m-r-s}\\
        =&q^{2l}*(cdh)^kafh(cdh)^{r-k-1}(bfg)^s(ceg)^{m-r-s}
    \end{split}
\end{equation*}
Then, applying the modular automorphism, we find that:
\begin{equation*}
    \begin{split}
        &q^{2l}\cdot h\left((cdh)^kafh(cdh)^{r-k-1}(bfg)^s(ceg)^{m-r-s}\right)\\
        =&q^{2l}\cdot h\left((cdh)^{r}(bfg)^{s+1}(ceg)^{m-r-s-1}\right)
    \end{split}
\end{equation*}

\hfill

\noindent Case 11): We have:
\begin{equation*}
    \begin{split}
        &(cdh)^kcah(cdh)^{r-k-1}(bfg)^lefg(bfg)^{s-l-1}(ceg)^{m-r-s}\\
        =&q^{2r-2k-1}\sum_{i=0}^l(q-q^{-1})^iq^{2i}{l\choose i}\\
        &*(cdh)^kafh(cdh)^{r-k-1}(bfg)^{s-1-i}(ceg)^{m-r-s+1+i}
    \end{split}
\end{equation*}
Applying the modular automorphism, we get:
\begin{equation*}
    \begin{split}
        &h\left((cdh)^kafh(cdh)^{r-k-1}(bfg)^{s-1-i}(ceg)^{m-r-s+1+i}\right)\\
        =&h\left((cdh)^{r}(bfg)^{s-i}(ceg)^{m-r-s+i}\right)
    \end{split}
\end{equation*}
Thus, we find:
\begin{equation*}
    \begin{split}
        &h\left((cdh)^kcah(cdh)^{r-k-1}(bfg)^lefg(bfg)^{s-l-1}(ceg)^{m-r-s}\right)\\
        =&q^{2r-2k-1}\sum_{i=0}^l(q-q^{-1})^iq^{2i}{l\choose i}h\left((cdh)^{r}(bfg)^{s-i}(ceg)^{m-r-s+i}\right)
    \end{split}
\end{equation*}

\hfill

\noindent Case 12): We have:
\begin{equation*}
    \begin{split}
        &(cdh)^kcah(cdh)^{r-k-1}(bfg)^{s}(ceg)^lfeg(ceg)^{m-r-s-l-1}\\
        =&q^{2l+2r-2k-2}\sum_{i=0}^s(q-q^{-1})^iq^{2i}{s\choose i}\\
        &*(cdh)^kafh(cdh)^{r-k-1}(bfg)^{s-i}(ceg)^{m-r-s+i}
    \end{split}
\end{equation*}
Applying the modular automorphism and Equation (\ref{eq:re}), we get:
\begin{equation*}
    \begin{split}
        &h\left((cdh)^kafh(cdh)^{r-k-1}(bfg)^{s-i}(ceg)^{m-r-s+i}\right)\\
        =&h\left((cdh)^{r}(bfg)^{s+1-i}(ceg)^{m-r-s+i-1}\right).
    \end{split}
\end{equation*}
Thus, we get:
\begin{equation*}
    \begin{split}
        &h\left((cdh)^kcah(cdh)^{r-k-1}(bfg)^{s}(ceg)^lfeg(ceg)^{m-r-s-l-1}\right)\\
        =&q^{2l+2r-2k-2}\sum_{i=0}^s(q-q^{-1})^iq^{2i}{s\choose i}*h\left((cdh)^{r}(bfg)^{s+1-i}(ceg)^{m-r-s+i-1}\right)
    \end{split}
\end{equation*}

\hfill

\noindent Case 13): We have:
\begin{equation*}
    \begin{split}
        &(cdh)^r(bfg)^kbcg(bfg)^lefg(bfg)^{s-k-l-2}(ceg)^{m-r-s}\\
        =&q^2\sum_{i=0}^l(q-q^{-1})^iq^{2i}{l\choose i}*(cdh)^r(bfg)^{s-1-i}(ceg)^{m-r-s+1+i}
    \end{split}
\end{equation*}

\hfill

\noindent Case 14): We have:
\begin{equation*}
    \begin{split}
        &(cdh)^r(bfg)^kbcg(bfg)^{s-k-1}(ceg)^lfeg(ceg)^{m-r-s-l-1}\\
        =&q^{2l+1}\sum_{i=0}^{s-k-1}(q-q^{-1})^iq^{2i}{s-k-1\choose i}*(cdh)^r(bfg)^{s-i}(ceg)^{m-r-s+i}
    \end{split}
\end{equation*}

\hfill

\noindent Case 15): We have:
\begin{equation*}
    \begin{split}
        &(cdh)^r(bfg)^s(ceg)^kcbg(ceg)^lfeg(ceg)^{m-r-s-l-k-2}\\
        =&q^{2l}*(cdh)^r(bfg)^{s+1}(ceg)^{m-r-s-1}
    \end{split}
\end{equation*}

\hfill

\noindent As we can see, standard monomials appearing in the linear relation derived from equation basis $(cdh)^r(bfg)^s(ceg)^{m-r-s}$ ($m-r-s\ge 1$) with comparing basis $(aek)^{m-1}bdk$ are of the form  $(cdh)^r(bfg)^{s+1-i}(ceg)^{m-r-s-1+i}$ with $0\le i\le s+1$. Thus, by our assumption, we can compute the Haar state of $(cdh)^r(bfg)^{s+1}(ceg)^{m-r-s-1}$ from this linear relation.

\subsection{The contribution of each case}
In this subsection, we sum over all possible values of index $l$ and $k$.\\
\\
\textbf{Case 1):}
\begin{equation*}
    \begin{split}
        &r\cdot h\left((cdh)^{r}(bfg)^{s+1}(ceg)^{m-r-s-1}\right)\\
        &-r(q-q^{-1})\cdot h\left((cdh)^{r}(bfg)^s(ceg)^{m-r-s}\right)
    \end{split}
\end{equation*}
\textbf{Case 2):}
\begin{equation*}
    \begin{split}
        s\cdot h\left((cdh)^{r}(bfg)^{s-1}(ceg)^{m-r-s+1}\right)
    \end{split}
\end{equation*}
\textbf{Case 3):}
\begin{equation*}
    \begin{split}
        &(m-r-s)\cdot h\left((cdh)^r(bfg)^{s+1}(ceg)^{m-r-s-1}\right)\\
        &-(q-q^{-1})(m-r-s)\cdot h\left((cdh)^r(bfg)^{s}(ceg)^{m-r-s}\right)
    \end{split}
\end{equation*}
\textbf{Case 4):} Denote
\begin{equation*}
    \mathcal{F}_1=h\left((cdh)^{r}(bfg)^{s+1}(ceg)^{m-r-s-1}\right)-(q-q^{-1})\cdot h\left((cdh)^{r}(bfg)^s(ceg)^{m-r-s}\right)
\end{equation*}
\begin{equation*}
    \begin{split}
        &\sum_{k=0}^{r-2}\sum_{l=0}^{r-2-k}q^{-2(l+1)}\cdot\mathcal{F}_1\\
        =&\left((r-1)\frac{q^{-2}}{1-q^{-2}}+\frac{q^{-2r+2}-1}{(1-q^2)^2}\right)\cdot\mathcal{F}_1
    \end{split}
\end{equation*}
\textbf{Case 5):}
\begin{equation*}
    \begin{split}
        &\sum_{k=0}^{r-1}\sum_{l=0}^{s-1}q^{-2(r-k)+1}\sum_{i=0}^{l+1}(q^{-1}-q)^i{l+1 \choose i}\cdot h\left((cdh)^{r}(bfg)^{s-i}(ceg)^{m-r-s+i}\right)\\
        =&\frac{q^{-2r+1}-q}{1-q^2}\sum_{i=1}^{s}(q^{-1}-q)^i{s+1\choose i+1}\cdot h\left((cdh)^{r}(bfg)^{s-i}(ceg)^{m-r-s+i}\right)\\
        &+s\frac{q^{-2r+1}-q}{1-q^2}\cdot h\left((cdh)^{r}(bfg)^{s}(ceg)^{m-r-s}\right)
    \end{split}
\end{equation*}
\textbf{Case 6):}
\begin{equation*}
    \begin{split}
        &\sum_{k=0}^{r-1}\sum_{l=0}^{m-r-s-1}q^{-2l-2r+2k}\sum_{i=0}^{s+1}(q^{-1}-q)^i{s+1\choose i}\cdot h\left((cdh)^{r}(bfg)^{s+1-i}(ceg)^{m-r-s-1+i}\right)\\
        =&\frac{q^{-2r}(1-q^{2r})(1-q^{-2(m-r-s)})}{(1-q^2)(1-q^{-2})}\sum_{i=0}^{s+1}(q^{-1}-q)^i{s+1\choose i}\cdot h\left((cdh)^{r}(bfg)^{s+1-i}(ceg)^{m-r-s-1+i}\right)
    \end{split}
\end{equation*}
\textbf{Case 7):}
\begin{equation*}
    \begin{split}
        &\sum_{k=0}^{s-2}\sum_{l=0}^{s-2-k}\sum_{i=0}^{l+1}(q^{-1}-q)^i{l+1\choose i}\cdot h\left((cdh)^r(bfg)^{s-1-i}(ceg)^{m-r-s+1+i}\right)\\
        =&\sum_{j=2}^{s}(q^{-1}-q)^{j-1}{s+1\choose j+1}\cdot h\left((cdh)^r(bfg)^{s-j}(ceg)^{m-r-s+j}\right)\\
        &+\frac{s(s-1)}{2}\cdot h\left((cdh)^r(bfg)^{s-1}(ceg)^{m-r-s+1}\right)\\
    \end{split}
\end{equation*}
\textbf{Case 8):}
\begin{equation*}
    \begin{split}
        &\sum_{k=0}^{s-1}\sum_{l=0}^{m-r-s-1}q^{-2l-1}\sum_{i=0}^{s-k}(q^{-1}-q)^i{s-k\choose i}\cdot h\left((cdh)^r(bfg)^{s-i}(ceg)^{m-r-s+i}\right)\\
        =&\frac{1-q^{-2(m-r-s)}}{q-q^{-1}}\sum_{i=1}^{s}(q^{-1}-q)^i{s+1\choose i+1}\cdot h\left((cdh)^r(bfg)^{s-i}(ceg)^{m-r-s+i}\right)\\
        &+s\frac{1-q^{-2(m-r-s)}}{q-q^{-1}}\cdot h\left((cdh)^r(bfg)^{s}(ceg)^{m-r-s}\right)
    \end{split}
\end{equation*}
\textbf{Case 9):} Using the same notation as in case 4)
\begin{equation*}
    \begin{split}
        &\sum_{k=0}^{m-r-s-2}\sum_{l=0}^{m-r-s-2-k}q^{-2l-2}\cdot \mathcal{F}_1\\
        =&\left(\frac{(m-r-s-1)q^{-2}}{1-q^{-2}}+\frac{q^{-2(m-r-s)+2}-1}{(1-q^2)^2}\right)\cdot \mathcal{F}_1
    \end{split}
\end{equation*}
\textbf{Case 10):}
\begin{equation*}
    \begin{split}
        &\sum_{k=0}^{r-2}\sum_{l=0}^{r-2-k}q^{2l}\cdot h\left((cdh)^{r}(bfg)^{s+1}(ceg)^{m-r-s-1}\right)\\
        =&\left(\frac{r-1}{1-q^2}+\frac{q^{2r}-q^2}{(1-q^2)^2}\right)\cdot h\left((cdh)^{r}(bfg)^{s+1}(ceg)^{m-r-s-1}\right)
    \end{split}
\end{equation*}
\textbf{Case 11):}
\begin{equation*}
    \begin{split}
        &\sum_{k=0}^{r-1}\sum_{l=0}^{s-1}q^{2r-2k-1}\sum_{i=0}^l(q-q^{-1})^iq^{2i}{l\choose i}\cdot h\left((cdh)^{r}(bfg)^{s-i}(ceg)^{m-r-s+i}\right)\\
        =&\frac{q^{2r-1}-q^{-1}}{1-q^{-2}}\sum_{i=0}^{s-1}(q-q^{-1})^iq^{2i}{s\choose i+1}\cdot h\left((cdh)^{r}(bfg)^{s-i}(ceg)^{m-r-s+i}\right)
    \end{split}
\end{equation*}
\textbf{Case 12):}
\begin{equation*}
    \begin{split}
        &\sum_{k=0}^{r-1}\sum_{l=0}^{m-r-s-1}q^{2l+2r-2k-2}\sum_{i=0}^s(q-q^{-1})^iq^{2i}{s\choose i}\cdot h\left((cdh)^{r}(bfg)^{s+1-i}(ceg)^{m-r-s+i-1}\right)\\
        =&\frac{(q^{2(m-s)}-q^{2r})(1-q^{-2r})}{(1-q^2)^2}\sum_{i=0}^s(q-q^{-1})^iq^{2i}{s\choose i}\cdot h\left((cdh)^{r}(bfg)^{s+1-i}(ceg)^{m-r-s+i-1}\right)
    \end{split}
\end{equation*}
\textbf{Case 13):}
\begin{equation*}
    \begin{split}
        &\sum_{k=0}^{s-2}\sum_{l=0}^{s-2-k}q^2\sum_{i=0}^l(q-q^{-1})^iq^{2i}{l\choose i}\cdot h\left((cdh)^r(bfg)^{s-1-i}(ceg)^{m-r-s+1+i}\right)\\
        =&q^2\sum_{j=1}^{s-1}(q-q^{-1})^{j-1}q^{2j-2}{s\choose j+1}\cdot h\left((cdh)^r(bfg)^{s-j}(ceg)^{m-r-s+j}\right)
    \end{split}
\end{equation*}
\textbf{Case 14):}
\begin{equation*}
    \begin{split}
        &\sum_{k=0}^{s-1}\sum_{l=0}^{m-r-s-1}q^{2l+1}\sum_{i=0}^{s-k-1}(q-q^{-1})^iq^{2i}{s-k-1\choose i}\cdot h\left((cdh)^r(bfg)^{s-i}(ceg)^{m-r-s+i}\right)\\
        =&\frac{q-q^{2(m-r-s)+1}}{1-q^2}\sum_{i=0}^{s-1}(q-q^{-1})^iq^{2i}{s\choose i+1}\cdot h\left((cdh)^r(bfg)^{s-i}(ceg)^{m-r-s+i}\right)\\
    \end{split}
\end{equation*}
\textbf{Case 15):}
\begin{equation*}
    \begin{split}
        &\sum_{k=0}^{m-r-s-2}\sum_{l=0}^{m-r-s-2-k}q^{2l}\cdot h\left((cdh)^r(bfg)^{s+1}(ceg)^{m-r-s-1}\right)\\
        =&\left(\frac{m-r-s-1}{1-q^2}+\frac{q^{2(m-r-s)}-q^2}{(1-q^2)^2}\right)\cdot h\left((cdh)^r(bfg)^{s+1}(ceg)^{m-r-s-1}\right)
    \end{split}
\end{equation*}

\subsection{Recursive relation for the general situation}
First, we discuss the general situation where $s\ge 2$, $r\ge 2$, and $m-r-s\ge 2$. In this situation, all 15 cases appear in the linear relation. \\
\\
The term $(cdh)^r(bfg)^{s+1}(ceg)^{m-r-s-1}$ appears in case 1), 3), 4), 6), 9), 10), 12), and 15). Summing the contributions from these cases, the coefficient of $(cdh)^r(bfg)^{s+1}(ceg)^{m-r-s-1}$ is:
\begin{equation*}
    \begin{split}
        &\frac{q^2(q^{m-s}-q^{s-m})^2}{(1-q^2)^2}
    \end{split}
\end{equation*}
The term $(cdh)^r(bfg)^{s}(ceg)^{m-r-s}$ appears in case 1), 3), 4), 5), 6), 8), 9), 11), 12), 14), and the right-hand-side of linear relation derived from equation basis $(cdh)^r(bfg)^{s}(ceg)^{m-r-s}$ with comparing basis $(aek)^{m-1}bdk$. Summing the contributions from these cases, the coefficient of \\ $(cdh)^r(bfg)^{s}(ceg)^{m-r-s}$ is:
\begin{equation*}
    \begin{split}
        &\frac{q+sq^{2m-2s+3}-(s+1)q^{-2m+2s+1}}{q^2-1}
    \end{split}
\end{equation*}
The term $(cdh)^r(bfg)^{s-1}(ceg)^{m-r-s+1}$ appears in case 2), 5), 6), 7), 8), 11), 12), 13), and 14). Notice that if we combine the contribution of case 2) and case 7), we get:
\begin{equation*}
    s+\frac{s(s-1)}{2}=\frac{s(s+1)}{2}={s+1\choose 2}
\end{equation*}
which corresponds to $j=1$ in the summation of case 7). Thus, we can treat $(cdh)^r(bfg)^{s-1}(ceg)^{m-r-s+1}$ in the same way as the general case $(cdh)^r(bfg)^{s-i}(ceg)^{m-r-s+i}$, $2\le i\le s-1$ which appears in case 5), 6), 7), 8), 11), 12), 13), and 14). Summing the contributions from these cases, the coefficient of $(cdh)^r(bfg)^{s-i}(ceg)^{m-r-s+i}$, $1\le i\le s-1$, is:
\begin{equation*}
    \begin{split}
        &(q^{-1}-q)^{i-1}{s+1\choose i+1}q^{-2m+2s}+(q-q^{-1})^{i-1}q^{2i-2}{s\choose i+1}q^{2m-2s+4}
    \end{split}
\end{equation*}
Notice that is we put $i=0$ in the above coefficient, we get:
\begin{equation*}
    \begin{split}
        \frac{(s+1)q^{-2m+2s}}{q^{-1}-q}+\frac{sq^{2m-2s+2}}{q-q^{-1}}=\frac{sq^{2m-2s+3}-(s+1)q^{-2m+2s+1}}{q^2-1}
    \end{split}
\end{equation*}
The term $(cdh)^r(ceg)^{m-r}$ appears in case 5), 6), 7), and 8). Summing the contributions from these cases, the coefficient of $(cdh)^r(ceg)^{m-r}$ is:
\begin{equation*}
    \begin{split}
        &(q^{-1}-q)^{s-1}q^{-2m+2s}
    \end{split}
\end{equation*}
Thus, the recursive relation for the general case is:
\begin{equation}
    \begin{split}
        &\frac{q^2(q^{m-s}-q^{s-m})^2}{(1-q^2)^2}\cdot h\left((cdh)^r(bfg)^{s+1}(ceg)^{m-r-s-1}\right)\\
        =&-\frac{q}{q^2-1}\cdot h\left((cdh)^r(bfg)^{s}(ceg)^{m-r-s}\right)\\
        &-\sum_{i=0}^{s-1}a_i\cdot h\left((cdh)^r(bfg)^{s-i}(ceg)^{m-r-s+i}\right)\\
        &-(q^{-1}-q)^{s-1}q^{-2m+2s}\cdot h\left((cdh)^r(ceg)^{m-r}\right)
    \end{split} \label{recursive:2}
\end{equation}
where
\begin{equation*}
    a_i=(q^{-1}-q)^{i-1}{s+1\choose i+1}q^{-2m+2s}+(q-q^{-1})^{i-1}q^{2i-2}{s\choose i+1}q^{2m-2s+4}
\end{equation*}
\subsection{Recursive relation for special situations}
Recall that the recursive relation is valid for $r\ge 1$ and $m-r-s\ge 1$.
\subsubsection{$s\ge 2$, $r=1$ and/or $m-r-s= 1$}
When $r=1$, case 4) and 10) cannot happen; when $m-r-s= 1$, case 9) and 15) cannot happen. If we substitute $=1$ into the contribution of case 4) and 10), we find that the contribution for these two cases is automatically zero. The same situation happens when we substitute $m-r-s= 1$ into the contribution of case 9) and 15). Hence, we conclude that the recursive relation Equation (\ref{recursive:2}) is still valid for the situation when $s\ge 2$, $r=1$ and/or $m-r-s= 1$.
\\
\subsubsection{$s=1$}
When $s=1$, case 7) and 13) cannot happen. If we apply the convention that when $a<b$ then ${a\choose b}=0$, then the contribution of case 7) and 13) is automatically zero when we substitute $s=1$. Hence, we conclude that the recursive relation Equation (\ref{recursive:2}) is still valid for the situation when $s=1$.
\subsubsection{$s=0$}
When $s=0$, only case 1), 3), 4), 6), 9), 10), 12), and 15) appears. The terms appearing the linear relation are $h\left((cdh)^r(bfg)(ceg)^{m-r-1}\right)$ and $h\left((cdh)^r(ceg)^{m-r}\right)$.\\
\\
Summing the contributions from these cases, the coefficient of $h\left((cdh)^r(bfg)(ceg)^{m-r-1}\right)$ is:
\begin{equation*}
    \begin{split}
        &\frac{q^2(q^m-q^{-m})^2}{(q^2-1)^2}
    \end{split}
\end{equation*}
The coefficient of $h\left((cdh)^r(ceg)^{m-r}\right)$ is:
\begin{equation*}
    \begin{split}
       &\frac{q-q^{1-2m}}{q^2-1}
    \end{split}
\end{equation*}
Thus, the corresponding recursive relation is:
\begin{equation*}
    \begin{split}
        &\frac{q^2(q^m-q^{-m})^2}{(q^2-1)^2}*h\left((cdh)^r(bfg)(ceg)^{m-r-1}\right)\\
        =&-\frac{q-q^{1-2m}}{q^2-1}*h\left((cdh)^r(ceg)^{m-r}\right)
    \end{split}
\end{equation*}
Notice that this is consistent with the recursive relation Equation (\ref{recursive:2}) for the general case when we substitute $s=0$. 
\section{The Haar state of standard monomials with segment $afh$ or $bdk$}\label{one_high_complex}
In this section, we comoute the Haar state of standard monomials in the form of $(afh)^w(bfg)^s(cdh)^r(ceg)^{m-w-r-s}$ or $(bdk)^w(bfg)^s(cdh)^r(ceg)^{m-w-r-s}$. If we apply the double flip anti-homomorphism $\omega$ to $(afh)^w(bfg)^s(cdh)^r(ceg)^{m-w-r-s}$, we get:
\begin{equation*}
    \begin{split}
        \omega\left((afh)^w(bfg)^s(cdh)^r(ceg)^{m-w-r-s}\right)&=(ceg)^{m-w-r-s}(cdh)^s(bfg)^r(bdk)^w\\
        &=(bfg)^r(cdh)^s(ceg)^{m-w-r-s}(bdk)^w
    \end{split}
\end{equation*}
Evaluating the Haar state on both sides and applying the modular automorphism, we get:
\begin{equation*}
    \begin{split}
        &h\left((afh)^w(bfg)^s(cdh)^r(ceg)^{m-w-r-s}\right)\\
        =&h\circ\omega\left((afh)^w(bfg)^s(cdh)^r(ceg)^{m-w-r-s}\right)\\
        =&h\left((bfg)^r(cdh)^s(ceg)^{m-w-r-s}(bdk)^w\right)\\
        =&h\left((bdk)^w\eta\left((bfg)^r(cdh)^s(ceg)^{m-w-r-s}\right)\right)\\
        =&h\left((bdk)^w(bfg)^r(cdh)^s(ceg)^{m-w-r-s}\right)
    \end{split}
\end{equation*}
Thus, we will develop an algorithm to compute the Haar state of\\ $(afh)^w(bfg)^s(cdh)^r(ceg)^{m-w-r-s}$ and the Haar state of \\$(bdk)^w(bfg)^r(cdh)^s(ceg)^{m-w-r-s}$ can be found by the equation above.

\subsection{Standard monomials in the form of\\ $(afh)^w(bfg)^s(cdh)^r(ceg)^{m-w-r-s}$ with ${m-w-r-s}\ge 1$}\label{case_with_ceg}
Before we introduce the algorithm, observe that:
\begin{equation*}
    \begin{split}
        ceg(afh)^k=q^{2k}*(afh)^kceg+\sum_{i=0}^{k-1}q^{2i}(1-q^2)*(afh)^ibfgcdh(afh)^{k-1-i}
    \end{split}
\end{equation*}
Using the same trick as in Section \ref{recuresive_cdhbfg} Equation (\ref{eq:re}), we get:
\begin{equation}
    (1-q^{2k})\cdot h\left((afh)^kceg\right)=\sum_{i=0}^{k-1}q^{2i}(1-q^2)\cdot h\left((afh)^ibfgcdh(afh)^{k-1-i}\right) \label{eq:2'}
\end{equation}
Notice that every monomial in the right-hand-side of Equation (\ref{eq:2'}) contains $k-1$ generator $a$ and no generator $k$. Thus, Theorem 1 e) tells us that the right-hand-side can be decomposed into a linear combination of standard monomials with at most $k-1$ generator $a$ and no generator $k$. This implies we can design a recursive algorithm on the number of $afh$ segments to compute the Haar state of $(afh)^w(cdh)^r(bfg)^s(ceg)^{m-w-r-s}$. 

\subsection{Standard monomials in the form of $(afh)^w(cdh)^r(bfg)^s$}\label{case_without_ceg}
 The strategy introduced in the subsection \ref{case_with_ceg} does not work for standard monomials in the form of $(afh)^w(cdh)^r(bfg)^s$ ($m-w-r-s=0$). In this case, we will use the linear relation derived from equation basis $(afh)^{w-1}(cdh)^{r+1}(bfg)^{m-r-w}$ with comparing basis $(aek)^{m-1}bdk$ to compute the Haar state of $(afh)^w(cdh)^r(bfg)^s$. For the detail of this construction, see Lu~\cite{lu2023} Section 2.3. In the following, we will show that if the Haar states of standard monomials in the form of $(afh)^w(cdh)^r(bfg)^s(ceg)^{m-w-r-s}$ with $m-w-r-s\ge 1$ and standard monomials with the number of $afh$ segments less than $w$ are known then we can compute the Haar state of $(afh)^w(cdh)^r(bfg)^s$ from this linear relation.\\
\\
Since the comparing basis is $(aek)^{m-1}bdk$, we have to consider left components in the following form:
\begin{enumerate}
    \item[1)] $(aek)^lbdk(aek)^{m-1-l}$
    \item[2)] $(aek)^lbek(aek)^kadk(aek)^{m-2-l-k}$
    \item[3)] $(aek)^ladk(aek)^kbek(aek)^{m-2-l-k}$ 
\end{enumerate} 
 When the left component is in the form of $(aek)^ibdk(aek)^{m-i-1}$, the corresponding right components are:
\begin{enumerate}
    \item[1)] $(afh)^ldch(afh)^{w-2-l}(cdh)^{r+1}(bfg)^{m-r-w}$ 
    \item[2)] $(afh)^{w-1}(cdh)^lfah(cdh)^{r-l}(bfg)^{m-r-w}$ 
    \item[3)] $(afh)^{w-1}(cdh)^{r+1}(bfg)^lecg(bfg)^{m-r-w-l-1}$ 
\end{enumerate}

\hfill

\noindent When the left component is in the form of $(aek)^kbek(aek)^ladk(aek)^{m-k-l-2}$, the corresponding right components are:
\begin{enumerate}
    \item[4)] $(afh)^ldfh(afh)^kach(afh)^{w-3-l-k}(cdh)^{r+1}(bfg)^{m-r-w}$ 
    \item[5)] $(afh)^ldfh(afh)^{w-2-l}(cdh)^kcah(cdh)^{r-k}(bfg)^{m-r-w}$ 
    \item[6)] $(afh)^ldfh(afh)^{w-2-l}(cdh)^{r+1}(bfg)^kbcg(bfg)^{m-r-w-k-1}$ 
    \item[7)] $(afh)^{w-1}(cdh)^lfdh(cdh)^kcah(cdh)^{r-l-k-1}(bfg)^{m-r-w}$ 
    \item[8)] $(afh)^{w-1}(cdh)^lfdh(cdh)^{r-l}(bfg)^kbcg(bfg)^{m-r-w-k-1}$ 
    \item[9)] $(afh)^{w-1}(cdh)^{r+1}(bfg)^lefg(bfg)^kbcg(bfg)^{m-r-w-l-k-2}$ 
\end{enumerate}

\hfill

\noindent When the left component is in the form of $(aek)^kadk(aek)^lbek(aek)^{m-k-l-2}$, the corresponding right components are:
\begin{enumerate}
    \item[10)] $(afh)^kach(afh)^ldfh(afh)^{w-3-l-k}(cdh)^{r+1}(bfg)^{m-r-w}$  
    \item[11)] $(afh)^kach(afh)^{w-2-k}(cdh)^lfdh(cdh)^{r-l}(bfg)^{m-r-w}$  
    \item[12)] $(afh)^kach(afh)^{w-2-k}(cdh)^{r+1}(bfg)^lefg(bfg)^{m-r-w-1-l}$  
    \item[13)] $(afh)^{w-1}(cdh)^kcah(cdh)^lfdh(cdh)^{r-1-k-l}(bfg)^{m-r-w}$ 
    \item[14)] $(afh)^{w-1}(cdh)^lcah(cdh)^{r-l}(bfg)^kefg(bfg)^{m-r-w-k-1}$ 
    \item[15)] $(afh)^{w-1}(cdh)^{r+1}(bfg)^kbcg(bfg)^lefg(bfg)^{m-r-w-2-k-l}$ 
\end{enumerate}
By counting the number of generator $a$ in these monomials, we know that case 1), 3), 4) to 6), 8), 9), 10) to 12), and 15) can be decomposed into a linear combination of standard monomials with at most $w-1$ $afh$ segments and we can compute their Haar states. The exceptions are case 2), 7), 13) and 14).

\hfill

\noindent For case 2), we have the following decomposition:
\begin{equation*}
    \begin{split}
        &(afh)^{w-1}(cdh)^lfah(cdh)^{r-l}(bfg)^{m-r-w}\\
        =&(afh)^{w-1}(cdh)^lafh(cdh)^{r-l}(bfg)^{m-r-w}\\
        &-(q-q^{-1})*(afh)^{w-1}(cdh)^{r+1}(bfg)^{m-r-w}\\
        =&(afh)^{w}(cdh)^{r}(bfg)^{m-r-w}\\
        &+(q^3-q)*\sum_{i=0}^{l-1}(afh)^{w-1}(cdh)^{l-1-i}afh(cdh)^{r-l+i}(bfg)^{m-r-w}ceg\\
        &-(q^2-q)l*(afh)^{w-1}(cdh)^{r}(bfg)^{m-r-w+1}\\
        &-(q-q^{-1})*(afh)^{w-1}(cdh)^{r+1}(bfg)^{m-r-w}.\\
    \end{split}
\end{equation*}
In the above equation, by applying Theorem 1 (e) to $(afh)^{w-1}(cdh)^{l-1-i}afh(cdh)^{r-l+i}(bfg)^{m-r-w}ceg$ we conclude that the decomposition of monomials in this form only contains standard monomials with at most $w$ $afh$ segments and at least one $ceg$ segment. Then, by our assumption, we can compute the Haar state of monomials in this form. Therefore, among all standard monomials appearing in the decomposition of case 2), only the Haar state of $(afh)^{w}(cdh)^{r}(bfg)^{m-r-w}$ is unknown. case 7) and 13) can be transformed into case 2) without generating new monomials. Hence, the conclusion of case 7) and 13) is the same as that of case 2).

\hfill

\noindent For case 14), there are $w$ generator $a$, $r+1$ generator $c$ and $m-r-w$ generator $g$ in these monomials. Thus, if we decompose these monomials into linear combinations of standard monomials, the standard monomials appearing in the linear combination contain at most $w$ generator $a$ and at least $r+1$ generator $c$ and at least $m-r-w$ generator $g$. Especially, there must be at least one $ceg$ segment in these standard monomials. By our assumption, the Haar states of these types of monomials are known. Hence, we can compute the Haar state of monomials in case 14).\\
\\
We conclude that $(afh)^w(cdh)^r(bfg)^{m-r-w}$ is the only standard monomial appearing in the linear relation whose Haar state is unknown. Next, we compute the coefficient of $(afh)^w(cdh)^r(bfg)^s(ceg)^{m-w-r-s}$ in the final linear relation. First assume that $r\ge 1$. The contribution of case 2) to the coefficient of \\$(afh)^w(cdh)^r(bfg)^s(ceg)^{m-w-r-s}$ in the final linear relation is:
$$r+1.$$
The contribution of case 7) to the coefficient of \\$(afh)^w(cdh)^r(bfg)^s(ceg)^{m-w-r-s}$ in the final linear relation is:
\begin{equation*}
    \begin{split}
        \sum_{l=0}^{r-1}\sum_{k=0}^{r-1-l}q^{-2k-2}=\sum_{l=0}^{r-1}\frac{q^{-2}-q^{-2(r-l)-2}}{1-q^{-2}}=\frac{rq^{-2}}{1-q^{-2}}-\frac{q^{-2r-2}}{1-q^{-2}}\frac{1-q^{2r}}{1-q^2}
    \end{split}
\end{equation*} 
The contribution of case 13) to the coefficient of \\$(afh)^w(cdh)^r(bfg)^s(ceg)^{m-w-r-s}$ in the final linear relation is:
\begin{equation*}
    \begin{split}
        &\sum_{l=0}^{r-1}\sum_{k=0}^{r-1-l}q^{2k}=\sum_{l=0}^{r-1}\frac{1-q^{2(r-l)}}{1-q^2}=\frac{r}{1-q^2}-\frac{q^{2r}}{1-q^2}\frac{1-q^{-2r}}{1-q^{-2}}
    \end{split}
\end{equation*}
Together, the coefficient of $(afh)^w(cdh)^r(bfg)^s(ceg)^{m-w-r-s}$ in the final linear relation is:
\begin{equation*}
    \begin{split}
    \frac{q^2(q^{-r-1}-q^{r+1})^2}{(1-q^2)^2}
    \end{split}
\end{equation*}
Then, in the case of $r=0$, case 7) and 13) cannot happen. The contribution from case 2) is $1$ and hence the coefficient of $(afh)^w(cdh)^r(bfg)^s(ceg)^{m-w-r-s}$ is $1$ in this situation which coincide with the result obtained by substituting $r=1$ into the coefficient of case $r\ge 1$. 
Since the coefficient is not identically zero for $r\ge 0$, we can compute the Haar state of $(afh)^w(cdh)^r(bfg)^{m-r-w}$, $r\ge 0$, from the linear relation.\\

\subsection{Recursive algorithm}
 Assume that we have computed the Haar state of standard monomials without high-complexity segment and standard monomials with $afh$ as the only high-complexity segments and the number of $afh$ segments does not exceed $w-1$. The recursive algorithm to compute the Haar state of standard monomials in the form of $(afh)^w(bfg)^s(cdh)^r(ceg)^{m-w-r-s}$ is:
 \begin{itemize}
     \item[i)] Compute the Haar state of monomials in the form of \\$(afh)^w(bfg)^s(cdh)^r(ceg)^{m-w-r-s}$ with $m-w-r-s\ge 1$ using Equation (\ref{eq:2'})
     \item[ii)] Compute the Haar state of monomials in the form of $(afh)^w(cdh)^r(bfg)^{m-r-w}$ using the linear relation derived from equation basis $(afh)^{w-1}(cdh)^{r+1}(bfg)^{m-r-w}$ with comparing basis $(aek)^{m-1}bdk$
 \end{itemize}
\section{The Haar state of standard monomials in general forms}
In this section, we provide an algorithm to compute the Haar state of standard monomials in the form of $(aek)^u(afh)^v(bdk)^w(bfg)^s(cdh)^r(ceg)^{m-u-v-w-r-s}$. We assume that the Haar states of standard monomials of order less or equal to $m-1$ are known. To simplify the argument, we will provide an algorithm to compute the Haar state of the general form of standard monomials regardless of whether the standard monomial is a monomial basis or not. The strategy is an induction on the number of high complexity segments in the general form of standard monomials. 
\subsection{Base case}\label{base_case}
The base case includes standard monomials in the form of $aek(bfg)^s(cdh)^r(ceg)^{m-1-r-s}$, $afh(bfg)^s(cdh)^r(ceg)^{m-1-r-s}$, and $bdk(bfg)^s(cdh)^r(ceg)^{m-1-r-s}$. The Haar state of $afh(bfg)^s(cdh)^r(ceg)^{m-1-r-s}$ and $bdk(bfg)^s(cdh)^r(ceg)^{m-1-r-s}$ are solved in Section \ref{one_high_complex}. The Haar state of $aek(bfg)^s(cdh)^r(ceg)^{m-1-r-s}$ is solved by the following identity:
$$h\left((bfg)^s(cdh)^r(ceg)^{m-1-r-s}\right)=h\left(D_q*(bfg)^s(cdh)^r(ceg)^{m-1-r-s}\right)$$
The left-hand-side is a monomial of order $m-1$ and by assumption, we know its Haar state. The right-hand-side is a linear combination of monomials of order $m$. Among these monomials, the only monomial with unknown Haar state value is $aek(bfg)^s(cdh)^r(ceg)^{m-1-r-s}$. Thus, we can find the Haar state of $aek(bfg)^s(cdh)^r(ceg)^{m-1-r-s}$ from this identity. Then, we find the Haar states for all standard monomials with one high complexity segment.

\subsection{Inductive steps}
Now, assume that we know the Haar state of all standard monomials in the form of $(aek)^u(afh)^v(bdk)^w(bfg)^s(cdh)^r(ceg)^{m-u-v-w-r-s}$ with $u+v+w\le n-1$. The Haar state of standard monomials in the form of\\ $(afh)^v(bfg)^s(cdh)^r(ceg)^{m-v-r-s}$ or $(bdk)^w(bfg)^s(cdh)^r(ceg)^{m-w-r-s}$ are solved in the Section \ref{base_case}.

\subsubsection{Monomials containing $ceg$ segments and no $aek$ segment}
To compute the Haar state of standard monomials with $n$ high complexity segments, we start with monomials in the form of \\$(afh)^v(bdk)^w(bfg)^s(cdh)^r(ceg)^{m-v-w-r-s}$ with $v,w\ge 1$ and $m-v-w-r-s\ge 1$. Since the counting matrix of monomial $afhbdkceg$ and $aekbfgcdh$ are the same, we can decompose $(afh)^v(bdk)^w(bfg)^s(cdh)^r(ceg)^{m-v-w-r-s}$ into a linear combination of standard monomials with at most $n-1$ high complexity segments in the following way:
\begin{enumerate}
    \item[1)] Rewrite $(afh)^v(bdk)^w(bfg)^s(cdh)^r(ceg)^{m-v-w-r-s}$ as a linear combination of $(afh)^{v-1}(bdk)^{w-1}\left[afhbdkceg\right](bfg)^s(cdh)^r(ceg)^{m-v-w-r-s-1}$ and other standard monomials with at most $n-1$ high complexity segments.
    \item[2)] Apply equation:
    \begin{equation}\label{afhbdkceg_equ}
        \begin{split}
            afhbdkceg =& q*aekbfgcdh+(1 - q^2)*aekbfgceg\\
        &+(1 - q^2)*aekcdhceg+(q^2 - 1)^2/q*aek(ceg)^2\\
        &+(1 - q^2)*afhbfgcdh+(q^3 - q)*afhbfgceg\\
        &+(q^3 - q)*afhcdhceg-(q^2 - 1)^2*afh(ceg)^2
        \end{split}
    \end{equation}
    to $(afh)^{v-1}(bdk)^{w-1}\left[afhbdkceg\right](bfg)^s(cdh)^r(ceg)^{m-v-w-r-s-1}$. Notice that each monomial on the right-hand-side contains only one high complexity segment. This means that $(afh)^v(bdk)^w(bfg)^s(cdh)^r(ceg)^{m-v-w-r-s}$ can be written as a linear combination of monomials with at most $n-1$ high complexity segments.
    \item[3)] Decompose these monomials with at most $n-1$ high complexity segments in to linear combinations of standard monomials. The standard monomials appearing in these decompositions contain at most $n-1$ high complexity segments as well. 
\end{enumerate}
Thus, by our assumption, we can compute the Haar state of standard monomials in the form of $(afh)^v(bdk)^w(bfg)^s(cdh)^r(ceg)^{m-v-w-r-s}$ with $v,w\ge 1$ and $m-v-w-r-s\ge 1$.

\subsubsection{Monomials without $ceg$ and $aek$ segment}
In this subsection, we compute the Haar state of monomials with $m-v-w-r-s=0$, i.e., monomials in the form of $(afh)^v(bdk)^w(bfg)^s(cdh)^{m-v-w-s}$ with $v,w\ge 1$. We will use the linear relation derived from equation basis \\$(afh)^v(bdk)^{w-1}(bfg)^{s+1}(cdh)^{m-v-w-s}$ and comparing basis $(aek)^{m-1}afh$. For the detail of this construction, see Lu~\cite{lu2023} Section 2.3.\\
\\
When the left components are in the form of $(aek)^{l}afh(aek)^{m-l-1}$, the corresponding right components are:
\begin{enumerate}
    \item[1)] $(afh)^lake(afh)^{v-1-l}(bdk)^{w-1}(bfg)^{s+1}(cdh)^{m-v-w-s}$ 
    \item[2)]  $(afh)^v(bdk)^lbgf(bdk)^{w-l-2}(bfg)^{s+1}(cdh)^{m-v-w-s}$
    \item[3)]  $(afh)^v(bdk)^{w-1}(bfg)^lbkd(bfg)^{s-l}(cdh)^{m-v-w-s}$
    \item[4)]  $(afh)^v(bdk)^{w-1}(bfg)^{s+1}(cdh)^lcge(cdh)^{m-v-w-s-1-l}$
\end{enumerate}
When the left components are in the form of $(aek)^kafk(aek)^laeh(aek)^{m-1-l-k}$, the corresponding right components are:
\begin{enumerate}
    \item[5)] $(afh)^kakh(afh)^lafe(afh)^{v-k-l-2}(bdk)^{w-1}(bfg)^{s+1}(cdh)^{m-v-w-s}$ 
    \item[6)] $(afh)^kakh(afh)^{v-1-k}(bdk)^lbdf(bdk)^{w-2-l}(bfg)^{s+1}(cdh)^{m-v-w-s}$
    \item[7)] $(afh)^kakh(afh)^{v-1-k}(bdk)^{w-1}(bfg)^lbfd(bfg)^{s-l}(cdh)^{m-v-w-s}$
    \item[8)] $(afh)^kakh(afh)^{v-1-k}(bdk)^{w-1}(bfg)^{s+1}(cdh)^lcde(cdh)^{m-v-w-s-1-l}$
    \item[9)] $(afh)^v(bdk)^kbgk(bdk)^lbdf(bdk)^{w-3-k-l}(bfg)^{s+1}(cdh)^{m-v-w-s}$ 
    \item[10)] $(afh)^v(bdk)^kbgk(bdk)^{w-2-k}(bfg)^lbfd(bfg)^{s-l}(cdh)^{m-v-w-s}$  
    \item[11)] $(afh)^v(bdk)^kbgk(bdk)^{w-2-k}(bfg)^{s+1}(cdh)^lcde(cdh)^{m-v-w-s-1-l}$
    \item[12)] $(afh)^v(bdk)^{w-1}(bfg)^kbkg(bfg)^lbfd(bfg)^{s-1-k-l}(cdh)^{m-v-w-s}$ 
    \item[13)] $(afh)^v(bdk)^{w-1}(bfg)^kbkg(bfg)^{s-k}(cdh)^lcde(cdh)^{m-v-w-s-1-l}$ 
    \item[14)] $(afh)^v(bdk)^{w-1}(bfg)^{s+1}(cdh)^kcgh(cdh)^lcde(cdh)^{m-v-w-s-2-k-l}$
\end{enumerate}
When the left components are in the form of $(aek)^kaeh(aek)^lafk(aek)^{m-1-l-k}$, the corresponding right components are:
\begin{enumerate}
    \item[15)] $(afh)^kafe(afh)^lakh(afh)^{v-k-l-2}(bdk)^{w-1}(bfg)^{s+1}(cdh)^{m-v-w-s}$ 
    \item[16)] $(afh)^kafe(afh)^{v-1-k}(bdk)^lbgk(bdk)^{w-2-l}(bfg)^{s+1}(cdh)^{m-v-w-s}$
    \item[17)] $(afh)^kafe(afh)^{v-1-k}(bdk)^{w-1}(bfg)^lbkg(bfg)^{s-l}(cdh)^{m-v-w-s}$
    \item[18)] $(afh)^kafe(afh)^{v-1-k}(bdk)^{w-1}(bfg)^{s+1}(cdh)^lcgh(cdh)^{m-v-w-s-1-l}$
    \item[19)] $(afh)^v(bdk)^kbdf(bdk)^lbgk(bdk)^{w-3-k-l}(bfg)^{s+1}(cdh)^{m-v-w-s}$ 
    \item[20)] $(afh)^v(bdk)^kbdf(bdk)^{w-2-k}(bfg)^lbkg(bfg)^{s-l}(cdh)^{m-v-w-s}$  
    \item[21)] $(afh)^v(bdk)^kbdf(bdk)^{w-2-k}(bfg)^{s+1}(cdh)^lcgh(cdh)^{m-v-w-s-1-l}$
    \item[22)] $(afh)^v(bdk)^{w-1}(bfg)^kbfd(bfg)^lbkg(bfg)^{s-1-k-l}(cdh)^{m-v-w-s}$ 
    \item[23)] $(afh)^v(bdk)^{w-1}(bfg)^kbfd(bfg)^{s-k}(cdh)^lcgh(cdh)^{m-v-w-s-1-l}$ 
    \item[24)] $(afh)^v(bdk)^{w-1}(bfg)^{s+1}(cdh)^kcde(cdh)^lcgh(cdh)^{m-v-w-s-2-k-l}$
\end{enumerate}
\paragraph{Analysis:}
The analysis of the 24 cases consists of 4 parts:
\begin{enumerate}
    \item[1)] Case 1), 5), 6), 15), and 16).
    \item[2)] Case 2), 4), 9), 10), 11), 14), 18), 19), 20), 21), 23) and 24).
    \item[3)] Case 3), 12), and 22).
    \item[4)] Case 7), 8), 13), and 17).
\end{enumerate}
We can decompose monomials of case 1) in the following way:
\begin{enumerate}
    \item[1)] Decompose $(afh)^lake(afh)^{v-1-l}(bdk)^{w-1}$ as a linear combination of standard monomials of order $v+w-1$. 
    \item[2)] Concatenate $(bfg)^s(cdh)^{m-v-w-s}$ at the end of every standard monomial appearing in the decomposition of $(afh)^lake(afh)^{v-1-l}(bdk)^{w-1}$. 
\end{enumerate}
Since low complexity segments commute with each other, we get a decomposition of monomials in case 1) after reorder the segments in the monomials in step 2). Since step 2) will not increase the number of high complexity segments in the decomposition, we conclude that the decomposition of monomials in case 1) only contains standard monomials with at most $v+w-1$ high complexity segments. Applying a similar argument to monomials of case 5), 6), 15), and 16), we get the same conclusion as monomials of case 1).\\
\\
For case 2), 4), 9), 10), 11), 14), 18), 19), 20), 21), 23) and 24), the sums of the number of generator $a$ and the number of generator $k$ are no more than $v+w-1$. Since each high-complexity segment contains at least one of $a$ or $k$, by Theorem 1 (e) we know that decomposition of these cases contains only standard monomials with at most $v+w-1$ high-complexity segments.\\
\\
For case 3), we apply the same argument as for case 1) and focus on the partial monomial $(afh)^v(bdk)^{w-1}(bfg)^lbkd(bfg)^{s-l}$. The partial monomial can be decomposed as:
\begin{equation*}
    \begin{split}
        &(afh)^v(bdk)^{w-1}(bfg)^lbkd(bfg)^{s-l}\\
        =&(afh)^v(bdk)^{w-1}(bfg)^lbdk(bfg)^{s-l}\\
        &-(q-q^{-1})*(afh)^v(bdk)^{w-1}(bfg)^{s+1}\\
        =&(afh)^v(bdk)^{w}(bfg)^{s}\\
        &-(q-q^{-1})*\sum_{i=0}^{l-1}(afh)^v(bdk)^{w-1}(bfg)^{l-1-i}bdk(bfg)^{s-l+i}ceg\\
        &+(q-q^{-1})l*(afh)^v(bdk)^{w-1}(bfg)^{s}cdh\\
        &-(q-q^{-1})*(afh)^v(bdk)^{w-1}(bfg)^{s+1}.\\
    \end{split}
\end{equation*}
By Theorem 1 (e), the decomposition of $(afh)^v(bdk)^{w-1}(bfg)^{l-1-i}bdk(bfg)^{s-l+i}ceg$ contains only standard monomials with at most $v+w$ high-complexity segments and at least one $ceg$ segment. Then, by our assumption, we can compute the Haar state of monomials in this form. Therefore, among all standard monomials appearing in the decomposition of case 3), only the Haar state of $(afh)^v(bdk)^{w}(bfg)^{s}(cdh)^{m-v-w-s}$ is unknown. Case 12) and 22) can be transformed in to case 3) without generating new monomials. Thus, these cases are essentially the same as case 3).\\
\\
For case 7), using a similar strategy as case 1), we focus on the partial monomials $(afh)^kakh(afh)^{v-1-k}(bdk)^{w-1}(bfg)^lbfd(bfg)^{s-l}$. Notice that the counting matrix of this monomial belongs to $A_3(v+w+s)$. The monomial can be decomposed into a linear combination of standard monomials of order $v+w+s$. By counting the number of generator $g$, we know that the decomposition of the partial monomial only contains standard monomials with at least $s$ low complexity segments with generator $g$. By counting the number of generator $a$ and $k$, we know that the standard monomials appearing in the decomposition contains at most $v+w$ high complexity segments. Thus, the decomposition of case 7) contains $(afh)^v(bdk)^w(bfg)^s(cdh)^{m-v-w-s}$ and other standard monomials that contains either $v+w$ high complexity segments and at least one $ceg$ segment or strictly least than $v+w$ high complexity segments. For case 8) and 17), by counting the number of generator $c$ or $g$ in the monomials and applying a similar argument to case 7), we conclude that standard monomials appearing in the decomposition of these cases contains at most $v+w-1$ high complexity segments. Notice that in case 8), we can switch $(bfg)^{s+1}$ with $(cdh)^lcde(cdh)^{m-v-w-s-1-l}$ without generating new monomials. For case 13), we can transform the case into the form $(afh)^v(bdk)^{w-1}(bfg)^kbk(bfg)^{s-k}(cdh)^ld(cdh)^{m-v-w-s-1-l}ceg$ without generating new monomials. Then, by counting the number of generator $a$ and $k$, we conclude that the decomposition of case 13) only contains standard monomials with at most $v+w$ high complexity segments and at least $1$ $ceg$ segment.
\paragraph{Conclusion:}
By previous analysis, the linear relation obtained from equation basis $(afh)^v(bdk)^{w-1}(bfg)^{s+1}(cdh)^{m-v-w-s}$ and comparing basis $(aek)^{m-1}afh$ contains the standard monomial $(afh)^v(bdk)^w(bfg)^s(cdh)^{m-v-w-s}$ and other standard monomials consisting of either $v+w$ high complexity segments and at least one $ceg$ segment or strictly least than $v+w$ high complexity segments. By our assumption, $(afh)^v(bdk)^w(bfg)^s(cdh)^{m-v-w-s}$ is the only standard monomial with unknown Haar state appearing in the linear relation.\\
\\
Next, we compute the coefficient of $(afh)^v(bdk)^w(bfg)^s(cdh)^{m-v-w-s}$ in the final linear relation. First, consider the case $s\ge 1$. Case 3), 7), 12), 22) contributes to the coefficient of $(afh)^v(bdk)^w(bfg)^s(cdh)^{m-v-w-s}$. The contribution of case 3) to the coefficient of $(afh)^v(bdk)^w(bfg)^s(cdh)^{m-v-w-s}$ is:
\begin{equation*}
    s+1.
\end{equation*}
The contribution of case 7) to the coefficient of $(afh)^v(bdk)^w(bfg)^s(cdh)^{m-v-w-s}$ is:
\begin{equation*}
    \begin{split}
        \sum_{k=0}^{v-1}\sum_{l=0}^{s}q^{-2l-2v+2k}=q^{-2v}\frac{1-q^{2v}}{1-q^2}\frac{1-q^{-2(s+1)}}{1-q^{-2}}
    \end{split}
\end{equation*}
The contribution of case 12) to the coefficient of $(afh)^v(bdk)^w(bfg)^s(cdh)^{m-v-w-s}$ is:
\begin{equation*}
    \begin{split}
        \sum_{k=0}^{s-1}\sum_{l=0}^{s-1-k}q^{-2l-2}=\sum_{k=0}^{s-1}\frac{q^{-2}-q^{-2(s-k)-2}}{1-q^{-2}}=\frac{sq^{-2}}{1-q^{-2}}-\frac{q^{-2s-2}}{1-q^{-2}}\frac{1-q^{2s}}{1-q^2}
    \end{split}
\end{equation*}
The contribution of case 22) to the coefficient of $(afh)^v(bdk)^w(bfg)^s(cdh)^{m-v-w-s}$ is:
\begin{equation*}
    \begin{split}
        \sum_{k=0}^{s-1}\sum_{l=0}^{s-1-k}q^{2l}=\sum_{k=0}^{s-1}\frac{1-q^{2(s-k)}}{1-q^2}=\frac{s}{1-q^2}-\frac{q^{2s}}{1-q^2}\frac{1-q^{-2s}}{1-q^{-2}}
    \end{split}
\end{equation*}
Together, the coefficient of $(afh)^v(bdk)^w(bfg)^s(cdh)^{m-v-w-s}$ in the final linear relation is:
\begin{equation*}
    \begin{split}
        &\frac{(q^{-2(s+1)}-1)(q^{-2v+2}-q^{2s+4})}{(1-q^2)^2}
    \end{split}
\end{equation*}
When $s=0$, case 12) and 22) cannot happen. Notice that the contributions of case 12) and 22) in case $s\ge 1$ are automatically $0$ if we substitute $s=0$ into these expressions. Hence, the coefficient expression obtained in case $s\ge 1$ is still valid for case $s=0$. Since the coefficient of $(afh)^v(bdk)^w(bfg)^s(cdh)^{m-v-w-s}$ is not identically zero, we can use this linear relation to compute the Haar state of $(afh)^v(bdk)^w(bfg)^s(cdh)^{m-v-w-s}$. 

\subsubsection{Monomials containing $aek$ segments}\label{13.2.3}
Finally, consider the general form $(aek)^u(afh)^v(bdk)^w(bfg)^s(cdh)^r(ceg)^{m-u-v-w-r-s}$ with $u\ge 1$ and $u+v+w=n$. We will apply a nested inductive argument on $u$. \\
\\
\textbf{Base case: $u=1$}\\
Consider the equality:
\begin{equation*}
    \begin{split}
        &(afh)^v(bdk)^w(bfg)^s(cdh)^r(ceg)^{m-1-v-w-r-s}\\
        =&D_q*(afh)^v(bdk)^w(bfg)^s(cdh)^r(ceg)^{m-1-v-w-r-s}\\
        =&aek(afh)^v(bdk)^w(bfg)^s(cdh)^r(ceg)^{m-1-v-w-r-s}\\
        &-q*(afh)^{v+1}(bdk)^w(bfg)^s(cdh)^r(ceg)^{m-1-v-w-r-s}\\
        &-q*bdk(afh)^v(bdk)^w(bfg)^s(cdh)^r(ceg)^{m-1-v-w-r-s}\\
        &+q^2*bfg(afh)^v(bdk)^w(bfg)^s(cdh)^r(ceg)^{m-1-v-w-r-s}\\
        &+q^2*cdh(afh)^v(bdk)^w(bfg)^s(cdh)^r(ceg)^{m-1-v-w-r-s}\\
        &-q^3*ceg(afh)^v(bdk)^w(bfg)^s(cdh)^r(ceg)^{m-1-v-w-r-s}
    \end{split}
\end{equation*}
Evaluate the Haar state on both sides and apply the modular automorphism, we get:
\begin{equation*}
    \begin{split}
        &(afh)^v(bdk)^w(bfg)^s(cdh)^r(ceg)^{m-1-v-w-r-s}\\
        =&aek(afh)^v(bdk)^w(bfg)^s(cdh)^r(ceg)^{m-1-v-w-r-s}\\
        &-q*(afh)^{v+1}(bdk)^w(bfg)^s(cdh)^r(ceg)^{m-1-v-w-r-s}\\
        &-q*bdk\left[(afh)^v(bdk)^w(bfg)^s(cdh)^r(ceg)^{m-1-v-w-r-s}\right]\\
        &+q^2*(afh)^v(bdk)^w(bfg)^{s+1}(cdh)^r(ceg)^{m-1-v-w-r-s}\\
        &+q^2*(afh)^v(bdk)^w(bfg)^s(cdh)^{r+1}(ceg)^{m-1-v-w-r-s}\\
        &-q^3*(afh)^v(bdk)^w(bfg)^s(cdh)^r(ceg)^{m-v-w-r-s}
    \end{split}
\end{equation*}
The left-hand-side of the equation is a monomial of order $m-1$ and we know its Haar state. The monomial $bdk\left[(afh)^v(bdk)^w(bfg)^s(cdh)^r(ceg)^{m-1-v-w-r-s}\right]$ decompose into a linear combination of standard monomials with at most $v+w$ high complexity segments. By counting the number of generator $a$ and $k$, we know that if a standard monomial with $v+w$ high complexity segments appears in the decomposition of $bdk\left[(afh)^v(bdk)^w(bfg)^s(cdh)^r(ceg)^{m-1-v-w-r-s}\right]$, it must contain $v$ $afh$ segments and $w$ $bdk$ segments. Therefore, we can compute the Haar state of $bdk\left[(afh)^v(bdk)^w(bfg)^s(cdh)^r(ceg)^{m-1-v-w-r-s}\right]$. Thus, the only monomial with unknown Haar state appearing in the equation is $aek(afh)^v(bdk)^w(bfg)^s(cdh)^r(ceg)^{m-1-v-w-r-s}$ and we can compute its Haar state using the equation. This finishes the base case.\\
\\
\textbf{Inductive steps:} $u=l$\\
Now, assume that the Haar states of standard monomials with $u\le l-1$ are known. To compute the Haar state of case $u=l$, we use the following equation:
\begin{equation*}
    \begin{split}
        &(aek)^{l-1}(afh)^v(bdk)^w(bfg)^s(cdh)^r(ceg)^{m-l-v-w-r-s}\\
        =&(aek)^{l-1}*D_q*(afh)^v(bdk)^w(bfg)^s(cdh)^r(ceg)^{m-l-v-w-r-s}\\
        =&(aek)^{l}(afh)^v(bdk)^w(bfg)^s(cdh)^r(ceg)^{m-l-v-w-r-s}\\
        &-q*(aek)^{l-1}(afh)^{v+1}(bdk)^w(bfg)^s(cdh)^r(ceg)^{m-l-v-w-r-s}\\
        &-q*(aek)^{l-1}*bdk*(afh)^v(bdk)^w(bfg)^s(cdh)^r(ceg)^{m-l-v-w-r-s}\\
        &+q^2*(aek)^{l-1}*bfg*(afh)^v(bdk)^w(bfg)^s(cdh)^r(ceg)^{m-l-v-w-r-s}\\
        &+q^2*(aek)^{l-1}*cdh*(afh)^v(bdk)^w(bfg)^s(cdh)^r(ceg)^{m-l-v-w-r-s}\\
        &-q^3*(aek)^{l-1}*ceg*(afh)^v(bdk)^w(bfg)^s(cdh)^r(ceg)^{m-l-v-w-r-s}
    \end{split}
\end{equation*}
For the non-standard monomials in the right-hand-side, we focus on the partial monomials after the $(aek)^{l-1}$ part. Applying a similar argument as in the base case, we conclude that these partial monomials can be decomposed as a linear combination of standard monomials of order $m-l+1$ with at most $v+w+1$ high complexity segments and when there are $v+w+1$ high complexity segments in the standard monomial, it must be $v$ $afh$ segments and $w+1$ $bdk$ segments. After concatenating $(aek)^{l-1}$ to the left of each standard monomial of order $m-l+1$ appearing in the decomposition of these partial monomials, we obtain the decomposition of the non-standard monomials in the right-hand-side. Besides $(aek)^{l}(afh)^v(bdk)^w(bfg)^s(cdh)^r(ceg)^{m-l-v-w-r-s}$, all the standard monomials appearing in the right-hand-side of equation contains at most $n$ high complexity segments and at most $l-1$ $aek$ segments. Hence, $(aek)^{l}(afh)^v(bdk)^w(bfg)^s(cdh)^r(ceg)^{m-l-v-w-r-s}$ is the only monomial with unknown Haar state value appearing in the equation and we can compute its Haar state from the equation. This finishes the inductive steps. Therefore, we are able to compute the Haar state of all monomials with $0\le u\le n$. This also means we can compute the Haar state of all monomial with $n$ high complexity segments and it finishes the inductive argument in this section.\\
\\
We have computed the Haar states of all standard monomials of order $m$.

\subsection{Recursive algorithm}
Assume that we have known the Haar state of standard monomials in the form of $(aek)^u(afh)^v(bdk)^w(bfg)^s(cdh)^r(ceg)^{m-u-v-w-r-s}$ with $u+v+w\le t-1$. Based on the induction, the recursive algorithm to compute the Haar state of standard monomials in the same form with $u+v+w=t$ is:
\begin{enumerate}
    \item[1)] Compute the Haar state of standard monomial in the form of \\$(afh)^v(bdk)^w(bfg)^s(cdh)^r(ceg)^{m-v-w-r-s}$ with $v,w,m-v-w-r-s\ge 1$ and $v+w=t$ using Equation (\ref{afhbdkceg_equ}).
    \item[2)] Compute the Haar state of standard monomials in the form of \\ $(afh)^v(bdk)^w(bfg)^s(cdh)^{m-v-w-s}$ with $v,w\ge 1, v+w=t$ using the linear relation derived from equation basis $(afh)^v(bdk)^{w-1}(bfg)^{s+1}(cdh)^{m-v-w-s}$ and comparing basis $(aek)^{m-1}afh$. 
    \item[3)] Compute the Haar state of standard monomials in the form of \\$(aek)^u(afh)^v(bdk)^w(bfg)^s(cdh)^r(ceg)^{m-u-v-w-r-s}$ with $u\ge 1, u+v+w=t$ using the inductive method in subsection \ref{13.2.3}.
\end{enumerate}
\section{Numerical results}
In this section, we present the Haar state of all standard monomials of order less than or equal to $3$. The Haar states for standard monomials of order $4$ and $5$ are available upon request.

\subsection{Haar state of standard monomials of order 1}
\begin{equation*}
    \begin{split}
        h(aek)&=\frac{(1-q^2)^2}{(1-q^4)(1-q^6)}\\
        h(afk)=h(bdk)&=\frac{(-q)(1-q^2)^2}{(1-q^4)(1-q^6)}\\
        h(bfg)=h(cdh)&=\frac{(-q)^2(1-q^2)^2}{(1-q^4)(1-q^6)}\\
        h(ceg)&=\frac{(-q)^3(1-q^2)^2}{(1-q^4)(1-q^6)}
    \end{split}
\end{equation*}
\subsection{Haar state of standard monomials of order 2}
\begin{equation*}
    \begin{split}
        h(aekaek)&=\frac{2q^8+q^4+1}{(q^2 + 1)^2(q^4 + 1)(q^2 - q + 1)^2(q^2 + q + 1)^2}\\
        h(aekafh)=h(aekbdk)&=\frac{-q(q^8 - q^6 + q^4 + 1)}{(q^2 + 1)^2(q^4 + 1)(q^2 - q + 1)^2(q^2 + q + 1)^2}\\
        h(aekbfg)=h(aekcdh)&=\frac{-q^2(q^6 - q^4 - 1)}{(q^2 + 1)^2(q^4 + 1)(q^2 - q + 1)^2(q^2 + q + 1)^2}\\
        h(aekceg)&=\frac{-q^3}{(q^2 + 1)^2(q^2 - q + 1)^2(q^2 + q + 1)^2}\\
        h(afhafh)=h(bdkbdk)&=\frac{q^2(q^4+1)}{(q^2 + 1)^2(q^2 - q + 1)^2(q^2 + q + 1)^2}\\
        h(afhbdk)&=\frac{-q^2(q^6 - q^4 - 1)}{(q^2 + 1)^2(q^4 + 1)(q^2 - q + 1)^2(q^2 + q + 1)^2}\\
        h(afhbfg)=h(bdkcdh)&=\frac{-q^3}{(q^2 + 1)^2(q^2 - q + 1)^2(q^2 + q + 1)^2}\\
        h(afhcdh)=h(bdkbfg)&=\frac{-q^3}{(q^2 + 1)^2(q^2 - q + 1)^2(q^2 + q + 1)^2}\\
    \end{split}
\end{equation*}
\begin{equation*}
    \begin{split}
        h(afhceg)=h(bfgcdh)=h(bdkceg)&=\frac{q^4}{(q^2 + 1)^2(q^4 + 1)(q^2 - q + 1)^2(q^2 + q + 1)^2}\\
        h(bfgbfg)=h(cdhcdh)&=\frac{q^4}{(q^2 + 1)(q^2 - q + 1)^2(q^2 + q + 1)^2}\\
        h(bfgceg)=h(cdhceg)&=\frac{-q^5}{(q^2 + 1)(q^4 + 1)(q^2 - q + 1)^2(q^2 + q + 1)^2}\\
        h(cegceg)&=\frac{q^6}{(q^4 + 1)(q^2 - q + 1)^2(q^2 + q + 1)^2}
    \end{split}
\end{equation*}
\subsection{Haar state of standard monomials of order 3}
\begin{equation*}
    \begin{split}
        &h\left((aek)^3\right)=\\
        &\frac{(q^{20} + 6q^{16} - 6q^{14} + 13q^{12} - 6q^{10} + 9q^8 - 2q^6 + 3q^4 - q^2 + 1)(1-q^2)}{(q^2 + 1)^2(q^4 + 1)^2(q^2 - q + 1)^2(q^2 + q + 1)^2(1-q^{10})}\\
        \\
        &h\left((aek)^2afh\right)=h\left((aek)^2bdk\right)=\\
        &\frac{-q(q^{18} - 2q^{14} + 7q^{12} - 7q^{10} + 8q^8 - 4q^6 + 3q^4 - q^2 + 1)(1-q^2)}{(q^2 + 1)^2(q^4 + 1)^2(q^2 - q + 1)^2(q^2 + q + 1)^2(1-q^{10})}\\
        \\
        &h\left((aek)^2bfg\right)=h\left((aek)^2cdh\right)=\\
        &\frac{-q^2(q^{16} - q^{14} - q^{12} + 3q^{10} - 5q^8 + 4q^6 - 3q^4 + q^2 - 1)(1-q^2)}{(q^2 + 1)^2(q^4 + 1)^2(q^2 - q + 1)^2(q^2 + q + 1)^2(1-q^{10})}\\
        \\
        &h\left((aek)^2ceg\right)=\\
        &\frac{-q^3(q^{14} + q^{10} + 3q^8 - 2q^6 + 3q^4 - q^2 + 1)(1-q^2)}{(q^2 + 1)^2(q^4 + 1)^2(q^2 - q + 1)^2(q^2 + q + 1)^2(1-q^{10})}\\
        \\
        &h\left(aek(afh)^2\right)=h\left(aek(bdk)^2\right)=\\
        &\frac{q^2(q^{18} - 2q^{16} + 2q^{14} + q^{12} - 3q^{10} + 7q^8 - 4q^6 + 4q^4 - q^2 + 1)(1-q^2)}{(q^2 + 1)^2(q^4 + 1)^2(q^2 - q + 1)^2(q^2 + q + 1)^2(1-q^{10})}\\
        \\
        &h\left(aekafhbdk\right)=\\
        &\frac{-q^2(q^{16} - q^{14} - 2q^{12} + 4q^{10} - 6q^8 + 5q^6 - 3q^4 + q^2 - 1)(1-q^2)}{(q^2 + 1)^2(q^4 + 1)^2(q^2 - q + 1)^2(q^2 + q + 1)^2(1-q^{10})}\\
        \\
        &h\left(aekafhbfg\right)=h\left(aekafhcdh\right)=\\
        &\frac{-q^3(q^{14} - q^{12} + 2q^8 - 3q^6 + 3q^4 - q^2 + 1)(1-q^2)}{(q^2 + 1)^2(q^4 + 1)^2(q^2 - q + 1)^2(q^2 + q + 1)^2(1-q^{10})}
    \end{split}
\end{equation*}
\begin{equation*}
    \begin{split}
        &h\left(aekafhceg\right)=\\
        &\frac{q^4(q^{10} - q^6 + 2q^4 - q^2 + 1)(1-q^2)}{(q^2 + 1)^2(q^4 + 1)^2(q^2 - q + 1)^2(q^2 + q + 1)^2(1-q^{10})}\\
        \\
        &h\left(aekbdkbfg\right)=h\left(aekbdkcdh\right)=\\
        &\frac{q^3(q^{16} - 2q^{14} + 2q^{12} - q^{10} - 2q^8 + 3q^6 - 3q^4 + q^2 - 1)(1-q^2)}{(q^2 + 1)^2(q^4 + 1)^2(q^2 - q + 1)^2(q^2 + q + 1)^2(1-q^{10})}\\
        \\
        &h\left(aekbdkceg\right)=\\
        &\frac{q^4(q^{12} - 2q^{10} + 4q^8 - 5q^6 + 4q^4 - 2q^2 + 1)(1-q^2)}{(q^2 + 1)(q^4 + 1)^2(q^2 - q + 1)^2(q^2 + q + 1)^2(1-q^{10})}\\
        \\
        &h\left(aekbfgbfg\right)=h\left(aekcdhcdh\right)=\\
        &\frac{-q^4(q^{10} - 2q^8 + 3q^6 - 3q^4 + q^2 - 1)(1-q^2)}{(q^2 + 1)(q^4 + 1)^2(q^2 - q + 1)^2(q^2 + q + 1)^2(1-q^{10})}\\
        \\
        &h\left(aekbfgcdh\right)=\\
        &\frac{q^4(1-q^2)^2(q^{10} + q^6 + 2q^4 + 1)}{(q^2 + 1)^2(q^4 + 1)^2(q^2 - q + 1)^2(q^2 + q + 1)^2(1-q^{10})}\\
        \\
        &h\left(aekbfgceg\right)=h\left(aekcdhceg\right)=\\
        &\frac{q^5(q^3 - q - 1)(q^3 - q + 1)(1-q^2)}{(q^2 + 1)(q^4 + 1)^2(q^2 - q + 1)^2(q^2 + q + 1)^2(1-q^{10})}\\
        \\
        &h\left(aek(ceg)^2\right)=\\
        &\frac{q^6(q^6 + q^4 + 1)(1-q^2)}{(q^2 + 1)(q^4 + 1)^2(q^2 - q + 1)^2(q^2 + q + 1)^2(1-q^{10})}\\
        \\
        &h\left((afh)^3\right)=h\left((bdk)^3\right)=\\
        &\frac{-q^3(q^8 - q^6 + 3q^4 - q^2 + 1)^2(1-q^2)}{(q^2 + 1)(q^4 + 1)^2(q^2 - q + 1)^2(q^2 + q + 1)^2(1-q^{10})}\\
        \\
        &h\left((afh)^2bdk\right)=h\left(afh(bdk)^2\right)=\\
        &\frac{q^3(q^{16} - 2q^{14} + 2q^{12} - 3q^8 + 4q^6 - 4q^4 + q^2 - 1)(1-q^2)}{(q^2 + 1)^2(q^4 + 1)^2(q^2 - q + 1)^2(q^2 + q + 1)^2(1-q^{10})}\\
        \\
    \end{split}
\end{equation*}
\begin{equation*}
    \begin{split}
        &h\left((afh)^2bfg\right)=h\left((bdk)^2cdh\right)=h\left((afh)^2cdh\right)=h\left((bdk)^2bfg\right)=\\
        &\frac{q^4(q^4 - q^2 + 1)(q^8 - q^6 + 3q^4 - q^2 + 1)(1-q^2)}{(q^2 + 1)(q^4 + 1)^2(q^2 - q + 1)^2(q^2 + q + 1)^2(1-q^{10})}\\
        \\
        &h\left((afh)^2ceg\right)=h\left((bdk)^2ceg\right)=\\
        &\frac{-q^5(q^4 - q^2 + 1)^2(1-q^2)}{(q^2 + 1)(q^4 + 1)^2(q^2 - q + 1)^2(q^2 + q + 1)^2(1-q^{10})}\\
        \\
        &h\left(afhbdkbfg\right)=h\left(afhbdkcdh\right)=\\
        &\frac{q^4(q^{12} - 3q^{10} + 5q^8 - 6q^6 + 5q^4 - 2q^2 + 1)(1-q^2)}{(q^2 + 1)(q^4 + 1)^2(q^2 - q + 1)^2(q^2 + q + 1)^2(1-q^{10})}\\
        \\
        &h\left(afhbdkceg\right)=\\
        &\frac{-q^5(1-q^2)^2(q^{10} + 2q^6 + q^4 + 1)}{(q^2 + 1)^2(q^4 + 1)^2(q^2 - q + 1)^2(q^2 + q + 1)^2(1-q^{10})}\\
        \\
        &h\left(afh(bfg)^2\right)=h\left(bdk(cdh)^2\right)=h\left(afh(cdh)^2\right)=h\left(bdk(bfg)^2\right)=\\
        &\frac{-q^5(q^8 - q^6 + 3q^4 - q^2 + 1)(1-q^2)}{(q^2 + 1)(q^4 + 1)^2(q^2 - q + 1)^2(q^2 + q + 1)^2(1-q^{10})}\\
        \\
        &h\left(afhbfgceg\right)=h\left(bdkcdhceg\right)=\\
        &\frac{q^6(q^4 - q^2 + 1)(1-q^2)}{(q^2 + 1)(q^4 + 1)^2(q^2 - q + 1)^2(q^2 + q + 1)^2(1-q^{10})}\\
        \\
        &h\left(afhcdhceg\right)=h\left(bdkbfgceg\right)=\\
        &\frac{q^6(q^4 - q^2 + 1)(1-q^2)}{(q^2 + 1)(q^4 + 1)^2(q^2 - q + 1)^2(q^2 + q + 1)^2(1-q^{10})}\\
        \\
        &h\left(afh(ceg)^2\right)=h\left(bdk(ceg)^2\right)=h\left(bfgcdhceg\right)=\\
        &\frac{-q^7(1-q^2)}{(q^2 + 1)(q^4 + 1)^2(q^2 - q + 1)^2(q^2 + q + 1)^2(1-q^{10})}\\
        \\
        &h\left((bfg)^3\right)=h\left((cdh)^3\right)=\\
        &\frac{q^6(q^8 - q^6 + 3q^4 - q^2 + 1)(1-q^2)}{(q^2 + 1)(q^4 + 1)^2(q^2 - q + 1)(q^2 + q + 1)(1-q^{10})}\\
        \\
    \end{split}
\end{equation*}
\begin{equation*}
    \begin{split}
        &h\left((bfg)^2cdh\right)=h\left(bfg(cdh)^2\right)=\\
        &\frac{q^6(q^4 - q^2 + 1)(1-q^2)}{(q^2 + 1)(q^4 + 1)^2(q^2 - q + 1)^2(q^2 + q + 1)^2(1-q^{10})}\\
        \\
        &h\left((bfg)^2ceg\right)=h\left((cdh)^2ceg\right)=\\
        &\frac{-q^7(q^4 - q^2 + 1)(1-q^2)}{(q^2 + 1)(q^4 + 1)^2(q^2 - q + 1)(q^2 + q + 1)(1-q^{10})}\\
        \\
        &h\left(bfg(ceg)^2\right)=h\left(cdh(ceg)^2\right)=\\
        &\frac{q^8(1-q^2)}{(q^2 + 1)(q^4 + 1)^2(q^2 - q + 1)(q^2 + q + 1)(1-q^{10})}\\
        \\
        &h\left((ceg)^3\right)=\frac{-q^9(1-q^2)}{(q^2 + 1)(q^4 + 1)^2(1-q^{10})}
    \end{split}
\end{equation*}

\paragraph*{Acknowledgement.}
The author was advised by Jeffrey Kuan from Texas A\&M University who acknowledges the support of
US NSF grant DMS-2000331 and Dr. Michael Brannan from the University of Waterloo. We sincerely thank their help and support during this research.

\bibliographystyle{plain}
\bibliography{reference}
\end{document}